\documentclass[a4paper]{pp-art}

\title[On a theorem of M.~Jodeit~Jr.]{On a theorem of M.~Jodeit~Jr.\break on pushforwards of Fourier multipliers}
\author{Patrick \textsc{Poissel}}
\address{Laboratoire de mathématiques de Besançon (UMR 6623), Université Marie \& Louis Pasteur, UFR Sciences et techniques, 16 route de Gray, 25030 Besançon CEDEX, France}
\email{patrick.poissel@math.cnrs.fr}
\date{Spring 2025\textendash Winter 2026}
\subjclass{43A22, 43A30 (primary), 43A35, 46L51, 46L52 (secondary)}
\keywords{non-commutative $L^p$ spaces, Fourier multipliers, positive definite distributions}

\usepackage{pp-thm}
\usepackage{tikz-cd}
\usepackage[french,english]{babel}

\usepackage{csquotes}
\usepackage[autolang=other, sorting=nyt, maxnames=9]{biblatex}
\AtEveryBibitem{\clearlist{language}}
\usepackage[colorlinks=true, linkcolor=blue, citecolor=red]{hyperref}
\AddToHook{begindocument/end}{
    \def\Re{\operatorname{Re}}
    
    \def\supp{\operatorname{supp}}
}

\begin{document}
\maketitle
\begin{abstract}
	A classical theorem of M.~\textsc{Jodeit}~Jr.\ implies that if a compactly supported distribution on $\mathbf{R}^d$ is the symbol of an $L^p(\mathbf{R}^d)$-$L^q(\mathbf{R}^d)$ Fourier multiplier, then its pushforward by the canonical homomorphism from $\mathbf{R}^d$ to $\mathbf{T}^d$ is the symbol of an $\ell^p(\mathbf{Z}^d)$-$\ell^q(\mathbf{Z}^d)$ Fourier multiplier. We generalise this result to the setting of locally compact groups, including those non-abelian, by characterising the continuous homomorphisms of locally compact groups by which, for every $p,q\in[1,\infty]$, the pushforward of a compactly supported symbol of an $L^p$-$L^q$ Fourier multiplier is a symbol of the same type as those which are \emph{open}. Motivated by a simple proof in the abelian case, we also investigate pushforwards of positive definite distributions.
\end{abstract}
\tableofcontents
\mainmatter
\section{Introduction}
The study of functorial properties of spaces of symbols of Fourier multipliers operators on $L^p$ spaces was initiated by K.~de~Leeuw \cite{deLeeuw} who obtained several pullback theorems by certain homomorphisms between the groups $\mathbf{Z}^d$, $\mathbf{R}^d$ and $\mathbf{T}^d$. His results were soon generalised to the setting of continuous homomorphisms of locally compact abelian groups by S.~Saeki \cite{Saeki} and N.~Lohoué \cite{LohoueEnssynthese, LohoueLasynthese, LohoueApprox}, who obtained the following general statement (see \autoref{section:fouriermultipliers} for details and general notations concerning Fourier multipliers).
\begin{theorem}[N.~Lohoué {\normalfont\cite[th.~I]{LohoueEnssynthese}\cite[th.~2]{LohoueApprox}}]\label{theorem:lohoué}
    Let $G$ and $H$ be two abelian locally compact groups and let $\pi$ be a continuous homomorphism from $G$ to $H$. Let $m$ be a sufficiently regular complex function on $H$ and let $p\in[1,\infty]$. For $m\circ\pi$ to be the symbol of an $L^p(\widehat{G{}})$-$L^p(\widehat{G{}})$ Fourier multiplier, it is sufficient that $m$ is the symbol of an $L^p(\widehat{H{}})$-$L^p(\widehat{H{}})$ Fourier multiplier, in which case
    \[
        \lVert m\circ\pi\rVert_{\operatorname{M}^{p,p}(G)}\leq\lVert m\rVert_{\operatorname{M}^{p,p}(H)}.
    \]
    If additionally $m$ is supported in the closure of $\pi(G)$ in $H$, then for $m\circ\pi$ to be the symbol of an $L^p(\widehat{G{}})$-$L^p(\widehat{G{}})$ Fourier multiplier, it is necessary and sufficient that $m$ is the symbol of an $L^p(\widehat{H{}})$-$L^p(\widehat{H{}})$ Fourier multiplier, in which case
    \[
        \lVert m\circ\pi\rVert_{\operatorname{M}^{p,p}(G)}=\lVert m\rVert_{\operatorname{M}^{p,p}(H)}.
    \]
\end{theorem}
See also the work of M.~Cowling \cite{CowlingExtensionperiodicity} on pullback theorems for symbols of Fourier multipliers of type $(p,q)$ with $p\neq q$.
\vskip\baselineskip
Following the recent trend of extending the study of Fourier multipliers to the setting of $L^p$ spaces of duals of possibly non-abelian locally compact groups (see the introduction of \cite{CaspersJanssensKrishnaswayUshaMiaskiwskyi} for a brief survey of the subject), pullback theorems for several classes of homomorphisms of non-abelian groups as well as the impossibility of a general pullback theorem for all homomorphisms were obtained by M.~Caspers, J.~Parcet, M.~Perrin and \'E.~Ricard in \cite{CaspersParcetPerrinRicard}. Subsequent improvements were then obtained by M.~Caspers, B.~Janssens, A.~Krishnaswamy-Usha and L.~Miaskiwskyi in \cite{CaspersJanssensKrishnaswayUshaMiaskiwskyi} and by B.~Janssens and B.~Oudejans in \cite{JanssensOudejans}. Contrary to pullback theorems, very few pushforward theorems exist in the literature, the main result being due to M.~Jodeit~Jr. \cite{Jodeit}.
\begin{theorem}[M.~Jodeit~Jr. {\normalfont\cite[th.~2.3]{Jodeit}}]
    Let $d\in\mathbf{N}$ and denote by $\pi$ the canonical homomorphism from $\mathbf{R}^d$ to $\mathbf{T}^d$. There exists a constant $C_d\geq 0$ such that for every $p,q\in[1,\infty]$, if $m$ is a sufficiently regular complex function on $\mathbf{R}^d$ supported in a fundamental cube and is the symbol of an $L^p(\mathbf{R}^d)$-$L^q(\mathbf{R}^d)$ Fourier multiplier, then its pushforward $\pi_*m$ by $\pi$ is the symbol of an $\ell^p(\mathbf{Z}^d)$-$\ell^q(\mathbf{Z}^d)$ Fourier multiplier and
    \[
        \lVert\pi_*m\rVert_{\operatorname{M}^{p,q}(\mathbf{T})}\leq C_d\lVert m\rVert_{\operatorname{M}^{p,q}(\mathbf{R}^d)}.
    \]%
\end{theorem}
Via the use of regularisation and partitions of unity with compactly supported smooth functions, one can then obtain the following pushforward theorem by the canonical homomorphism from $\mathbf{R}^d$ to $\mathbf{T}^d$ for general compactly supported symbols on $\mathbf{R}^d$.
\begin{corollary}\label{corollary:pushforwardRdTd}
    For every compact subset $K$ of $\mathbf{R}^d$, there exists a constant $C_K\geq 0$ such that for every $p,q\in[1,\infty]$, if $m$ is a distribution on $\mathbf{R}^d$ supported in $K$ and is the symbol of an $L^p(\mathbf{R}^d)$-$L^q(\mathbf{R}^d)$ Fourier multiplier, then $\pi_*m$ is the symbol of an $\ell^p(\mathbf{Z}^d)$-$\ell^q(\mathbf{Z}^d)$ Fourier multiplier and
    \[
        \lVert\pi_*m\rVert_{\operatorname{M}^{p,q}(\mathbf{T})}\leq C_K\lVert m\rVert_{\operatorname{M}^{p,q}(\mathbf{R}^d)}.
    \]
\end{corollary}
Several interesting generalisations of this pushforward theorem for linear operators and maximal functions of weak type $(p,q)$ were later obtained by P.~Auscher and M.~J.~Carro in \cite{AuscherCarro}, where however it is still only question of the canonical homomorphism from $\mathbf{R}^d$ to $\mathbf{T}^d$. It is therefore desirable to obtain pushforward theorems for large classes of continuous homomorphisms of locally compact groups.
\vskip\baselineskip

If $\pi$ is a continuous group homomorphism from a locally compact group $G$ to another locally compact group $H$, we shall say that \emph{the conclusion of Jodeit's theorem holds for $\pi$} if for every compact subset $K$ of $G$, there exists a constant $C_K\geq 0$ such that for every $p,q\in[1,\infty]$, if a distribution $m$ on $G$ supported $K$ is the symbol of an $L^p(\widehat{G{}})$-$L^q(\widehat{G{}})$ Fourier multiplier, then $\pi_*m$ is the symbol of an $L^p(\widehat{H{}})$-$L^q(\widehat{H{}})$ Fourier multiplier and
    \[
        \lVert\pi_*m\rVert_{\operatorname{M}^{p,q}(H)}\leq C_K\lVert m\rVert_{\operatorname{M}^{p,q}(G)}.
    \]
The main result of the present work is the following simple characterisation of continuous homomorphisms of locally compact groups for which the conclusion of Jodeit's theorem holds.
\begin{theorem}\label{theorem:ncjodeit}
    Let $G$ and $H$ be two locally compact groups and let $\pi$ be a continuous homomorphism from $G$ to $H$. For the conclusion of Jodeit's theorem to hold for $\pi$, it is necessary and sufficient that $\pi$ is open.
\end{theorem}
A large class of open homomorphisms is the class of quotient homomorphisms. In the abelian setting, the conclusion of Jodeit's theorem for quotient homomorphisms can be reformulated as follows: if $G$ is an abelian locally compact group, $N$ is a closed subgroup of $G$ and $k$ is a tempered continuous function on $\widehat{G{}}$ with compactly supported Fourier transform such that the integral operator $f\mapsto k*f$ on $\widehat{G{}}$ is of type $(p,q)$, then the integral operator $f\mapsto k|_{N^\perp}*f$ on $N^\perp$ is also of type $(p,q)$. The proofs of the original theorem of M.~Jodeit and of its variants given by P.~Auscher and M.~Carro rely heavily on the integral operator point of view. Our proof of \autoref{theorem:ncjodeit} (\autoref{section:generalpush}) shall take another approach and have all the computations be performed on the multiplication side of the duality. We shall also present evidence suggesting that, given a locally compact group $G$ and a closed normal subgroup $N$ of $G$, it may not even  be possible to come up with a general definition for <<~the quantum subgroup $N^\perp$ of $\widehat{G{}}$~>> (\autoref{section:functorialprop}).
\section{Terminology and notations}
\subsection{Integration}
Throughout this work, if $G$ is a locally compact group, we denote by $e_G$ its identity element, by $\mu_G$ a left Haar measure on $G$ and by $\Delta_G$ its modular function. Convolutions of functions on $G$ and the Lebesgue spaces $L^p(G)$ should be understood to be taken with respect to $\mu_G$. If $p\in[1,\infty]$,  its conjugate exponent will be denoted by $p'$. The inner product on $L^2(G)$ will be denoted by $\langle\cdot|\cdot\rangle$.

If $N$ is a closed \emph{normal} subgroup of $G$, we assume that $\mu_G$, $\mu_{G/N}$, $\mu_N$ are chosen so that
\[
	\int_Gf(x)\,d\mu_G(x)=\int_{G/N}\left(\int_Nf(xy)\,d\mu_N(y)\right)d\mu_{G/N}(xN)
\]
for every $f\in L^1(G)$. If $G'$ is an \emph{open} subgroup of $G$, we assume that $\mu_{G'}$ and $\mu_G$ are chosen so that $\mu_G$ induces $\mu_{G'}$ on $G'$.

And finally, if $f$ is a complex function on $G$, it will be convenient to denote by $f^\sharp$ and $f^\flat$ the functions given by
\[
	f^\sharp(s)=\Delta_G(s^{-1})\overline{f(s^{-1})}\qquad\text{and}\qquad f^\flat(s)=\overline{f(s^{-1})}
\]
respectively for every $s\in G$ as is commonplace when dealing with the Tomita-Takesaki theory, so that
\[
	\langle f*g|h\rangle=\langle g|f^\sharp*h\rangle=\langle f|h*g^\flat\rangle\qquad\text{and}\qquad\langle f|g\rangle=\langle g^\flat|f^\sharp\rangle=\langle g^\sharp|f^\flat\rangle
\]
for every compactly supported continuous $f,g,h$.
\subsection{Distributions}
Let $G$ be a locally compact group. We will make constant use of the space of compactly supported smooth functions $\mathscr{D}(G)$, the space of smooth functions $\mathscr{E}(G)$, the space of distributions $\mathscr{D}'(G)$, the space of compactly supported distributions $\mathscr{E}'(G)$ and the Schwartz space $\mathscr{S}(G)$ and the space of tempered distributions $\mathscr{S}'(G)$ if $G$ is abelian, described by F.~Bruhat in \cite{Bruhat} (see also \cite{Wawrzynczyk}).

If $S$ and $T$ are distributions on $G$ and it makes sense to evaluate $S$ at $T$, we will denote by $\langle S,T\rangle$ the value of $S$ at $T$. Similarly, if it makes sense to multiply $S$ and $T$, we will denote by $S\cdot T$ or more simply by $ST$ their product. The operation $f\mapsto\bar f$ on functions can be extended continuously to distributions with as usual
\[
	\langle\overline T,f\rangle=\overline{\langle T,\bar f\rangle}
\]
for every $T\in\mathscr{D}'(G)$ and every $f\in\mathscr{D}(G)$. Whenever it makes sense, we shall write $\langle S|T\rangle$ for $\langle S,\overline T\rangle$. The operations $f\mapsto f^\sharp$ and $f\mapsto f^\flat$ can also be continuously extended to distributions with
\[
	\langle T^\sharp|f\rangle=\langle f^\flat|T\rangle\qquad\text{and}\qquad\langle T^\flat|f\rangle=\langle f^\sharp|T\rangle
\]
for every $T\in\mathscr{D}'(G)$ and every $f\in\mathscr{D}(G)$.

Any complex Radon measure $\mu$ on $G$ will be identified with the distribution $f\mapsto\int f\,d\mu$. Any (equivalence class modulo equality locally $\mu_G$-almost everywhere of a) complex locally $\mu_G$-integrable function $f$ will be identified with the complex Radon measure on $G$ with density $f$ with respect to $\mu_G$.
\vskip\baselineskip
We will denote by $A(G)$ the Fourier algebra of $G$ as defined by P.~Eymard in \cite{Eymard}. We recall that $\mathscr{D}(G)$ is a dense subspace of $A(G)$ into which its canonical injection is continuous \cite[prop.~3.26]{Eymard}. Hence, $A(G)$ is the closure of $\mathscr{D}(G)$ in the Fourier-Stieltjes algebra $B(G)$ and the dual $A'(G)$ of $A(G)$ can be and will be identified with the subspace of $\mathscr{D}'(G)$ consisting of all the distributions which are continuous when $\mathscr{D}(G)$ is endowed with the topology induced by $A(G)$. Since $A(G)$ consists of locally $\mu_G$-integrable functions, it can also be identified with a space of distributions on $G$. Identifying both $A'(G)$ and $A(G)$ with subspaces of $\mathscr{D}'(G)$ makes $(A'(G),A(G))$ into a compatible pair of Banach spaces \cite[n\textsuperscript{o}\nobreak\,2.3]{BerghLofstrom}.
\subsection{Pushforward and pullback operators}
Let $G$ and $H$ be two locally compact groups and let $\pi$ be a \emph{continuous} homomorphism from $G$ to $H$. If $h$ is a smooth function on $H$, then $h\circ\pi$ is a smooth function on $G$ and the linear map $h\mapsto h\circ\pi$ from $\mathscr{E}(H)$ to $\mathscr{E}(G)$ is continuous \cite[prop.~8]{Bruhat}. We will refer to this map as the \emph{pullback operator by $\pi$} and denote it by $\pi^*$ and refer to its transpose as the \emph{pushforward operator by $\pi$} and denote it by $\pi_*$.

Depending on the quality of $\pi$, the operators $\pi^*$ and $\pi_*$ may be extended to spaces of functions and distributions larger than just $\mathscr{E}(H)$ or $\mathscr{E}'(G)$. For example, the various functorial properties of $\mathscr{D}(\cdot)$ established by F.~Bruhat in \cite[\textsection\nobreak\,7 and \textsection\nobreak\,8]{Bruhat} can be be summarised as follows.
\begin{proposition}
	\label{proposition:pushopenD} Suppose that $\pi$ is open and denote by $N$ its kernel. The pushforward operator $\pi_*$ realises a strict homomorphism of topological vector spaces from $\mathscr{D}(G)$ onto the closed subspace of $\mathscr{D}(H)$ consisting of all the compactly supported smooth functions on $H$ supported in the image of $\pi$. Moreover, the Haar measures $\mu_G$, $\mu_N$ and $\mu_H$ can be chosen so that $\pi^*(\pi_*g)=g*(\mu_N)^\flat$ for every $g\in\mathscr{D}(G)$.	
\end{proposition}
Hence, when $\pi$ is open, the pullback operator $\pi^*$ may be extended by transposition into a continuous linear operator from $\mathscr{D}'(H)$ to $\mathscr{D}'(G)$.
\section{Generalities on Fourier multipliers}\label{section:fouriermultipliers}
Here, we collect a few classical results from commutative and non-commuta\-tive Fourier analysis in order to properly motivate the introduction of our spaces of symbols of Fourier multipliers of type $(p,q)$ with $p\neq q$.
\subsection{Classical Fourier multipliers}
Let $G$ be an \emph{abelian} locally compact group, denote by $\widehat{G{}}$ its dual and by $\mathscr{F}_G$ the Fourier transform. If $m\in\mathscr{S}'(G)$, the \emph{Fourier multiplier with symbol $m$} is the continuous linear operator $m(D)$ from $\mathscr{S}(\widehat{G{}})$ to $\mathscr{S}'(\widehat{G{}})$ such that
\[
	(m(D)\circ\mathscr{F}_G)(f)=\mathscr{F}_G(m\cdot f)
\]
for every $f\in\mathscr{S}(G)$. If $p,q\in[1,\infty]$, we will denote by $\operatorname{M}^{p,q}(G)$ the set of all $m\in\mathscr{S}'(G)$ for which $m(D)$ is of type $(p,q)$, i.e.\ for which there exists a constant $C\geq 0$ such that
\[
	\lVert m(D)\hat f\rVert_{L^q(\widehat{G{}})}\leq C\lVert\hat f\rVert_{L^p(\widehat{G{}})}
\]
for every $f\in\mathscr{S}(G)$. The group $G$ being abelian, the space $\operatorname{M}^{1,1}(G)$ is none other than the Fourier-Stieltjes algebra $B(G)$. If we let
\[
	\lVert T\rVert_{\operatorname{M}^{p,q}(G)}=\sup\left\{\lvert\langle T,fg\rangle\rvert~\middle\vert~f,g\in\mathscr{S}(G)\text{ and }\lVert\hat f\rVert_{L^p(\widehat{G{}})},\lVert\hat g\rVert_{L^{q\mathrlap{'}}(\widehat{G{}})}\leq 1\right\},
\]
for every $T\in\mathscr{S}'(G)$, then $\lVert\cdot\rVert_{\operatorname{M}^{p,q}(G)}$ is an absolutely homogeneous subadditive lower semi-continuous function from $\mathscr{S}'(G)$ to $[0,+\infty]$ and
\[
	\operatorname{M}^{p,q}(G)=\left\{T\in\mathscr{S}'(G)~\middle\vert~\lVert T\rVert_{\operatorname{M}^{p,q}(G)}<+\infty\right\},
\]
showing that $\operatorname{M}^{p,q}(G)$ is a linear subspace of $\mathscr{S}'(G)$, countable union of closed convex balanced subsets of $\mathscr{S}'(G)$ and on which $\lVert\cdot\rVert_{\operatorname{M}^{p,q}(G)}$ is a norm.
\vskip\baselineskip
Since $\mathscr{D}(G)$ is dense in $\mathscr{S}(G)$ and since the topology of $\mathscr{S}(G)$ is finer than those given by the norms $f\mapsto\lVert\hat f\rVert_{L^p(\widehat{G{}})}$ and $g\mapsto\lVert\hat g\rVert_{L^{q\mathrlap{'}}(\widehat{G{}})}$, $\mathscr{D}(G)$ could replace $\mathscr{S}(G)$ in the formula defining $\lVert\cdot\rVert_{\operatorname{M}^{p,q}(G)}$ and we can define
\[
	\lVert T\rVert_{\operatorname{M}^{p,q}(G)}=\sup\left\{\lvert\langle T,fg\rangle\rvert~\middle\vert~f,g\in\mathscr{D}(G)\text{ and }\lVert\hat f\rVert_{L^p(\widehat{G{}})},\lVert\hat g\rVert_{L^{q\mathrlap{'}}(\widehat{G{}})}\leq 1\right\}
\]
when $T$ is merely in $\mathscr{D}'(G)$ rather than $\mathscr{S}'(G)$. This extension does not however produce any new Fourier multipliers of type $(p,q)$.
\begin{proposition}\label{proposition:Mpqtempered}
	Let $p,q\in[1,\infty]$ and let $m\in\mathscr{D}'(G)$ be such that $\lVert m\rVert_{\operatorname{M}^{p,q}(G)}$ is finite. Then, the distribution $m$ is tempered.
\end{proposition}
    The proof below is quite standard and follows the same scheme as that of \cite[chap.~VI, \textsection\nobreak\,8, th.~XXV]{SchwartzTD}.
\begin{proof}
    It is sufficient to prove that the restriction of $m$ to any compactly generated open subgroup of $G$ is tempered. Let $U$ be a compactly generated open subgroup of $G$. It is clear that the restriction $m|_U$ of $m$ to $U$ satisfies $\lVert m|_U\rVert_{\operatorname{M}^{p,q}(U)}<+\infty$. Denote by $\Phi$ the continuous linear operator from $L^p(\widehat{U})$ ($\mathscr{C}_0(\widehat{U})$ if $p=\infty$) to $L^q(\widehat{U})$ such that
    \[
        \Phi(\mathscr{F}_Uf)=\mathscr{F}_U(m|_U\cdot f)
    \]
    for every $f\in\mathscr{D}(U)$. Let $(\varepsilon_n)_{n\in\mathbf{N}}$ be a sequence of elements of $\mathscr{D}(\widehat{U})$, bounded in $L^1(\widehat{U})$ and which converges to $\delta_{e_{\widehat{U\null}}}$ in $\mathscr{E}'(e_{\widehat{U}})$. It is clear that
    \[
        \Phi(\varepsilon_n*\hat f)=\Phi(\varepsilon_n)*\hat f
    \]
    for every $n\in\mathbf{N}$ and every $f\in\mathscr{D}(U)$ and this relation can be extended by continuity to every $f\in\mathscr{S}'(U)$ such that $\hat f\in L^p(\widehat{U})\cap L^{q\mathrlap{'}}(\widehat{U})$ ($\hat f\in \mathscr{C}_0(\widehat{U})\cap L^{q\mathrlap{'}}(\widehat{U})$ if $p=\infty$) and in particular to every $f\in\mathscr{S}'(U)$ such that $\hat f$ is a compactly supported continuous function. Since $U$ is compactly generated, there exists a $d\in\mathbf{N}$ such that its dual $\widehat{U}$ is locally isomorphic to $\mathbf{R}^d$. There is therefore a distribution $X$ on $\widehat{U}$ supported in $\{\cramped{e_{\widehat{U\null}}}\}$ as well as two compactly supported continuous complex functions $\hat g$ and $\hat h$ on $\widehat{U}$ such that
    \[
        X*\hat g+\hat h=\delta_{e_{\widehat{U\null}}}
    \]
    (if $d=0$, simply take $X=\hat g=0$ and $\hat h=\delta_{e_{\widehat{U\null}}}$, and if $d\neq 0$, see \cite[formula~(VI,\nobreak\,6;\nobreak\,22)]{SchwartzTD}). The relations
    \[
        \Phi(\varepsilon_n)=\Phi(\varepsilon_n)*X*\hat g+\Phi(\varepsilon_n)*\hat h=X*\Phi(\varepsilon_n*\hat g)+\Phi(\varepsilon_n*\hat h)
    \]
    for every $n\in\mathbf{N}$ then show that the sequence $(\Phi(\varepsilon_n))_{n\in\mathbf{N}}$ of elements of $L^q(\widehat{U})$ is bounded in $\mathscr{S}'(\widehat{U})$ and therefore has a limit point $k$ in $\mathscr{S}'(\widehat{U})$. It is then clear that $\Phi(\hat f)=k*\hat f$ for every $f\in\mathscr{D}(U)$ and then that $m|_U$ is the inverse Fourier transform of $k$ and is therefore tempered. 
\end{proof}
\subsection{Lebesgue spaces on the dual of a locally compact group}
Here, we recall a few facts about non-commutative $L^p$ spaces on von~Neumann algebras, including those non-tracial, and Fourier analysis on locally compact groups, including those non-abelian. The reader interested in the fine details shall consult \cite{IzumiConstruction, IzumiBilin, Caspers}.
\vskip\baselineskip
Let $\mathscr{M}$ be a von~Neumann algebra and denote by $\mathscr{M}_\star$ its predual. Let $\nu$ be a faithful normal semi-finite \emph{weight} on $\mathscr{M}$ and as usual let
\[
	\mathfrak{n}=\{a\in\mathscr{M}~\vert~\nu(a^*a)<\infty\}.
\]
There exists a unique way to make $(\mathscr{M},\mathscr{M}_\star)$ into a compatible pair of Banach spaces for which
\[
	\mathscr{M}\cap\mathscr{M}_\star=\{a\in\mathfrak{n}~\vert~\exists a_\star\in\mathscr{M}_\star,~\forall b\in\mathfrak{n},~\langle b^*,a_\star\rangle=\nu(ab^*)\}.
\]
The compatible pair $(\mathscr{M},\mathscr{M}_\star)$ obtained this way is the one considered by H.~Izumi with parameter $-1/2$ (see \cite[n\textsuperscript{o}\nobreak\,2]{IzumiConstruction} and \cite[prop.~2.14]{Caspers}). Let us endow $\mathscr{M}$ with its natural operator space structure and $\mathscr{M}_\star$ with the operator space structure \emph{opposite} of that induced by the dual of $\mathscr{M}$ \cite[n\textsuperscript{o}\nobreak\,2.10]{PisierIntrotoOS}. For every $p\in\mathopen]1,\infty\mathclose[$, we then define $L^p(\nu)$ to be the operator space $(\mathscr{M},\mathscr{M}_\star)_{1/p}$ obtained by using the lower complex interpolation method \cite[n\textsuperscript{o}\nobreak\,2.7]{PisierIntrotoOS} on the pair $(\mathscr{M},\mathscr{M}_\star)$ with index $1/p$. We also let $L^\infty(\nu)=\mathscr{M}$ and $L^1(\nu)=\mathscr{M}_\star$.

It can be checked that $L^\infty(\nu)\cap L^2(\nu)=\mathfrak{n}$ and that $\lVert a\rVert_{L^2(\nu)}=\nu(a^*a)^{1/2}$ for every $a\in\mathfrak{n}$ (see \cite[lemma~5.4]{IzumiBilin} or \cite[th.~3.5]{Caspers}). $L^2(\nu)$ is therefore a Hilbert space isometrically isomorphic to the Hilbert space of the Gelfand-Naimark-Segal construction of $\nu$.

If $p$ and $p'$ are conjugate exponents in $[1,\infty]$, the sesquilinear functional $\langle\cdot|\cdot\rangle$ on $\mathscr{M}\cap\mathscr{M}_\star\times\mathscr{M}\cap\mathscr{M}_\star$ given by
\[
	\langle a|b\rangle=\nu(b^*a)
\]
can be continuously extended to $L^p(\mu)\times L^{p\mathrlap{'}}(\nu)$ in a way such that $b\mapsto\langle\cdot|b\rangle$ realises an antilinear isometry from $L^{p\mathrlap{'}}(\nu)$ to the dual of $L^p(\nu)$ and is surjective whenever $p<\infty$ \cite[th.~2.5 and th.~6.1]{IzumiBilin}. In particular, if $1<p<\infty$, then $L^p(\nu)$ is reflexive.
\vskip\baselineskip
Suppose now that $G$ is a locally compact group, that $\mathscr{M}$ is the von~Neumann algebra generated by the left regular representation $\lambda_G$ of $G$ on $L^2(G)$ and that $\nu$ is the Plancherel weight associated to $\mu_G$, i.e.\ for every $a\in\mathscr{M}_+$, if there is a $f\in L^2(G)$ and a constant $C\geq 0$ such that
\[
	\lVert f*g\rVert_{L^2(G)}\leq C\lVert g\rVert_{L^2(G)}\qquad\text{and}\qquad a(g)=f^\sharp*f*g
\]
for every $g\in\mathscr{D}(G)$, then $\nu(a)=\lVert f\rVert_{L^2(G)}^2=(f^\sharp*f)(e_G)$, and if there is not, then $\nu(a)=+\infty$. We will prefer to denote by $L^p(\widehat{G{}})$ the space $L^p(\nu)$.

We now describe the Fourier transform. There is a unique isometric linear isomorphism from $A'(G)$ to $\mathscr{M}$ which is continuous for the topologies $\sigma(A'(G),A(G))$ and $\sigma(\mathscr{M},\mathscr{M}_\star)$ and which maps $T\in A'(G)$ to the operator $\widehat{T}\in\mathscr{M}$ given by
\[
	\widehat{T}(f)=T*f
\]
for every $f\in\mathscr{D}(G)$ \cite[th.~3.10 and prop.~3.27]{Eymard}. This isometry coincides on the algebra of bounded Radon measures with the homomorphism $\mu\mapsto\lambda_G(\mu)$ and in particular with the Fourier transform on $L^1(G)$ as described by M.~Caspers in \cite[th.~5.2 (ii)]{Caspers}. From the definition of the Plancherel weight, it can be seen that
\[
	L^\infty(\widehat{G{}})\cap L^2(\widehat{G{}})=\{\hat f~\vert~f\in A'(G)\cap L^2(G)\}
\]
and that $\lVert\hat f\rVert_{L^2(\widehat{G{}})}=\lVert f\rVert_{L^2(G)}$ for every $f\in A'(G)\cap L^2(G)$, hence the existence of an isometric isomorphism from $L^2(G)$ to $L^2(\widehat{G{}})$ which coincides with $f\mapsto\hat f$ on $A'(G)\cap L^2(G)$ and therefore is none other than the Fourier transform on $L^2(G)$ described by M.~Caspers in \cite[th.~5.2 (i)]{Caspers}. Similarly, from the above description of $L^\infty(\widehat{G{}})\cap L^2(\widehat{G{}})$, it can be seen that
\[
	L^\infty(\widehat{G{}})\cap L^1(\widehat{G{}})=\{\hat f~\vert~f\in A'(G)\cap A(G)\}
\]
and that $\lVert\hat f\rVert_{L^1(\widehat{G{}})}=\lVert f\rVert_{A(G)}$, hence the existence of an isometric isomorphism from $A(G)$ to $L^1(\widehat{G{}})$ which coincides with $f\mapsto\hat f$ on $A'(G)\cap A(G)$. We call \emph{Fourier transform} of $G$ and denote by $\mathscr{F}_G$ the isometric isomorphism from $A'(G)+A(G)$ to $L^\infty(\widehat{G{}})+L^1(\widehat{G{}})$ such $\mathscr{F}_G(T)=\widehat{T}$ for every $T\in A'(G)$. The operator space closure of $\mathscr{F}_G(\mathscr{D}(G))$ in $L^\infty(\widehat{G{}})$ will be denoted by $\mathscr{C}_0(\widehat{G{}})$. If no confusion is possible, we will sometimes denote by $\widehat{T}$ the image by $\mathscr{F}_G$ of an arbitrary $T\in A'(G)+A(G)$.
\vskip\baselineskip
The sesquilinear form $\langle\cdot|\cdot\rangle$ on $L^p(\widehat{G{}})\times L^{p\mathrlap{'}}(\widehat{G{}})$ has a simple interpretation. It is easy to see that if $T\in\mathscr{F}_G^{-1}(L^p(\widehat{G{}}))$ and $f\in\mathscr{D}(G)$, then $\langle\widehat{T}|\hat f\rangle=\langle T|f\rangle$. Hence, we have the
\begin{proposition}\label{proposition:characLp}
	Let $p\in[1,\infty\mathclose[$. The space $\mathscr{F}_G(\mathscr{D}(G))$ is a dense subspace of $L^p(\widehat{G{}})$ and for a distribution $T$ on $G$ to be the inverse Fourier transform of an element of $L^{p\mathrlap{'}}(\widehat{G{}})$, it is necessary and sufficient that the quantity
	\[
		\sup\left\{\lvert\langle T,f\rangle\rvert~\middle\vert~f\in\mathscr{D}(G)\text{ and }\lVert\hat f\rVert_{L^p(\widehat{G{}})}\leq 1\right\}
	\]
	is finite, in which case it is equal to the norm in $L^{p\mathrlap{'}}(G)$ of $\widehat{T}$.
\end{proposition}
\begin{proof}
	Since $T\mapsto\langle\mathop{\cdot}|T\rangle$ realises an isometric antiisomorphism between $L^{p\mathrlap{'}}(\widehat{G{}})$ and the dual of $L^p(\widehat{G{}})$, the first part of the proposition can be obtained by a simple bipolarity argument while the second part is an easy consequence of the first.
\end{proof}
\subsection{Non-commutative Fourier multipliers}
Let $G$ be a \emph{possibly non-abelian} locally compact group. While no theory of tempered distributions on locally compact quantum groups seems to have yet entered the canons, the temperence theorem given by \autoref{proposition:Mpqtempered} motivates a very naïve definition of symbols of Fourier multipliers of type $(p,q)$ on $G$. We shall denote by $\operatorname{M}^{p,q}(G)$ (\emph{resp.}\ $\operatorname{cbM}^{p,q}(G)$) the space of all $m\in\mathscr{D}'(G)$ such that there exists a bounded (\emph{resp.}\ completely bounded) operator $m(D)$ from $L^p(\widehat{G{}})$ ($\mathscr{C}_0(\widehat{G{}})$ if $p=\infty$) to $L^q(\widehat{G{}})$ such that
\[
    m(D)(\mathscr{F}_Gf)=\mathscr{F}_G(mf)
\]
for every $f\in\mathscr{D}(G)$ and endow it with the norm deduced from the operator norm (\emph{resp.}\ completely bounded operator norm). Hence, by \autoref{proposition:characLp}, if we let
\[
	\lVert T\rVert_{\operatorname{M}^{p,q}(G)}=\sup\left\{\lvert\langle T,fg\rangle\rvert~\middle\vert~f,g\in\mathscr{D}(G)\text{ and }\lVert\hat f\rVert_{L^p(\widehat{G{}})},\lVert\hat g\rVert_{L^{q\mathrlap{'}}(\widehat{G{}})}\leq 1\right\}
\]
for every $T\in\mathscr{D}'(G)$, then $\operatorname{M}^{p,q}(G)$ is the subspace of $\mathscr{D}'(G)$ on which $\lVert\cdot\rVert_{\operatorname{M}^{p,q}(G)}$ is finite and the map $m\mapsto\lVert m\rVert_{\operatorname{M}^{p,q}(G)}$ on $\operatorname{M}^{p,q}(G)$ is the norm of $\operatorname{M}^{p,q}(G)$. Similarly, if we let
\[
    \lVert T\rVert_{\operatorname{cbM}^{p,p}(G)}=\lVert T\otimes 1\rVert_{\operatorname{M}^{p,p}(G\times\operatorname{SU}(2))}
\]
for every $T\in\mathscr{D}'(G)$, then $\operatorname{cbM}^{p,p}(G)$ is the subspace of $\mathscr{D}'(G)$ on which $\lVert\cdot\rVert_{\operatorname{cbM}^{p,p}(G)}$ is finite and the map $m\mapsto\lVert m\rVert_{\operatorname{cbM}^{p,p}(G)}$ on $\operatorname{cbM}^{p,p}(G)$ is the norm of $\operatorname{cbM}^{p,p}(G)$.
\begin{proposition}
	The topology of $\operatorname{M}^{p,q}(G)$ (\emph{resp.}\ $\operatorname{cbM}^{p,q}(G)$) is finer than that induced by $\mathscr{D}'(G)$ and $\operatorname{M}^{p,q}(G)$ (\emph{resp.}\ $\operatorname{cbM}^{p,q}(G)$) is a Banach space.
\end{proposition}
\begin{proof}
	Since the function $T\mapsto\lVert T\rVert_{\operatorname{M}^{p,q}(G)}$ from $\mathscr{D}'(G)$ to $[0,+\infty]$ is lower semi-continuous, the closed unit ball of $\operatorname{M}^{p,q}(G)$ is complete for the uniformity induced by $\mathscr{D}'(G)$ and to prove the proposition, it is therefore sufficient to prove it is \emph{bounded} in $\mathscr{D}'(G)$. This is elementary however because $\mathscr{D}(G)$ is barrelled and because if $f\in\mathscr{D}(G)$, then there exists a $g\in\mathscr{D}(G)$ such that $f=fg$ and therefore such that
	\[
		\lvert\langle m,f\rangle\rvert=\lvert\langle m,fg\rangle\rvert\leq\lVert\hat f\rVert_{L^p(\widehat{G{}})}\lVert\hat g\rVert_{L^{q\mathrlap{'}}(\widehat{G{}})}
	\]
	for every $m\in\mathscr{D}'(G)$ such that $\lVert m\rVert_{\operatorname{M}^{p,q}(G)}\leq 1$. The case of $\operatorname{cbM}^{p,q}_c(G)$ can be dealt with similarly.
\end{proof}
If $K$ is a compact subset of $G$, we shall denote by $\operatorname{M}^{p,q}_K(G)$ the set of all $m\in \operatorname{M}^{p,q}(G)$ supported in $K$. $\operatorname{M}^{p,q}_K(G)$ is a closed linear subspace of $\operatorname{M}^{p,q}(G)$; we therefore endow it with the induced topology. We shall denote by $\operatorname{M}^{p,q}_c(G)$ and by $A_c(G)$ the set of all compactly supported $m\in \operatorname{M}^{p,q}(G)$. The set $\operatorname{M}^{p,q}_c(G)$ is a linear subspace of $\operatorname{M}^{p,q}(G)$ and is the union of the directed family of the spaces $\operatorname{M}^{p,q}_K(G)$; we shall therefore endow it with the final locally convex topology with respect to the canonical injections of the spaces $\operatorname{M}^{p,q}_K(G)$ into it. We define similarly the space $\operatorname{cbM}^{p,q}_K(G)$, $\operatorname{cbM}^{p,q}_c(G)$, $A_K(G)$ and $A_c(G)$.

\begin{proposition}
	The locally convex space $\operatorname{M}^{p,q}_c(G)$ (\emph{resp.}\ $\operatorname{cbM}^{p,q}_c(G)$) is Hausdorff, complete and bornological.
\end{proposition}
\begin{proof}
	The bornologicity of $\operatorname{M}^{p,q}_c(G)$ is obvious from its definition so we shall only prove that $\operatorname{M}^{p,q}_c(G)$ is Hausdorff and complete. Let $G'$ be an open, \emph{countable at infinity} subgroup of $G$ and if $U$ is a left translate of $G'$, let
	\[
		\operatorname{M}^{p,q}_U(G)=\{m\in\operatorname{M}^{p,q}_c(G)~\vert~\operatorname{supp}(m)\subseteq U\}=\bigcup_{\substack{K\text{ compact}\\ \text{subset of }U}}\operatorname{M}^{p,q}_K(G)
	\]
	and endow it with the final locally convex topology with respect to the canonical injection of the spaces $\operatorname{M}^{p,q}_K(G)$ into it as $K$ ranges over the set of all compact subsets of $U$. The spaces $\operatorname{M}^{p,q}_U(G)$ are either all Banach spaces or all strict inductive limits of sequences of spaces, which are complete. Since each ${\boldsymbol 1}_U$ lies in the Fourier-Stieltjes algebra $B(G)$, the space $\operatorname{M}^{p,q}_c(G)$ can be checked to be the locally convex direct sum of the complete spaces $\operatorname{M}^{p,q}_U(G)$ and is therefore complete.
	
	The case of $\operatorname{cbM}^{p,q}_c(G)$ can be dealt with similarly.
\end{proof}

While it may not be true that $\operatorname{M}^{1,1}(G)=B(G)$ when $G$ is not abelian, we still have the following result.
\begin{lemma}\label{lemma:M11cisAc}
	Let $K$ be a compact subset of $G$. The topological vector spaces $\operatorname{M}_K^{1,1}(G)$, $\operatorname{cbM}^{1,1}_K(G)$ and $A_K(G)$ are identical and for every $f\in \operatorname{M}^{1,1}_K(G)$ and every non-negligible compact subset $L$ of $G$,
	\[
	    \lVert f\rVert_{\operatorname{M}^{1,1}(G)}\leq\lVert f\rVert_{\operatorname{cbM}^{1,1}(G)}\leq\lVert f\rVert_{A(G)}\leq\biggl(\frac{\mu_G(KL)}{\mu_G(L)}\biggr)^{1/2}\lVert f\rVert_{\operatorname{M}^{1,1}(G)}.
	\]
	Consequently, the topological vector spaces $\operatorname{M}^{1,1}_c(G)$, $\operatorname{cbM}^{1,1}_c(G)$ and $A_c(G)$ are identical.
\end{lemma}
\begin{proof}
	It is clear that $A_K(G)$ is a subspace of $\operatorname{M}^{1,1}_K(G)$ into which its canonical injection is contractive. Conversely, if $g$ is the function $\mu_G(L)^{-1}{\boldsymbol 1}_{KL}*({\boldsymbol 1}_L)^\flat$, then $g=1$ on $K$ and $g\in A(G)$ with $\lVert g\rVert_{A(G)}\leq (\mu_G(KL)/\mu_G(L))^{1/2}$. And of course,
	\[
		f=fg\in A_K(G)\quad\text{and}\quad\lVert f\rVert_{A(G)}=\lVert fg\rVert_{A(G)}\leq\lVert f\rVert_{\operatorname{M}^{1,1}(G)}\lVert g\rVert_{A(G)}
	\]
	for every $f\in \operatorname{M}^{1,1}_K(G)$.
\end{proof}
\begin{remark}
	It can be checked that $A_c(G)$ is a subspace of $A'(G)\cap A(G)$ into which its canonical injection is continuous and that $\mathscr{D}(G)$ is a dense subspace of $A_c(G)$ into which its canonical injection is continuous. Hence, the spaces of symbols $\operatorname{M}^{p,q}(G)$ would have essentially been the same had we used G.~Gaudry's <<~quasi-measures~>> \cite{GaudryQuasimeasures} instead of distributions and $A_c(G)$ instead of $\mathscr{D}(G)$.
\end{remark}
\section{The general pushforward theorem}\unskip\label{section:generalpush}
\subsection{Reduction to the case \texorpdfstring{$p=q=1$}{p=q=1}}
In order to prove \autoref{theorem:ncjodeit}, we first prove that the conclusion of Jodeit's pushforward theorem for arbitrary exponents $p$ and $q$ is equivalent to the same conclusion but merely for the pair of exponents $(1,1)$. We shall actually prove a slightly better characterisation.
\begin{lemma}\label{lemma:pq11}
    Let $G$ and $H$ be two locally compact groups and let $\pi$ be a continuous homomorphism from $G$ to $H$. The following conditions are equivalent:
    \begin{enumerate}
        \item The conclusion of Jodeit's theorem holds for $\pi$.
        \item The pushforward operator $\pi_*$ maps $A_c(G)$ in $A(H)$
        \item The pushforward operator $\pi_*$ continuously maps $A_c(G)$ in $A(H)$.
    \end{enumerate}
\end{lemma}
\begin{proof}
    The implication (i)~$\Rightarrow$~(ii) is obvious since $A_c(\cdot)=\operatorname{M}^{1,1}_c(\cdot)$ (\autoref{lemma:M11cisAc}) while the implication (ii)~$\Rightarrow$~(iii) is an easy consequence of the closed graph theorem since for every compact subset $K$ of $G$, the topology of $A_K(G)\times A(H)$ is finer than the one induced by $\mathscr{E}'(G)\times\mathscr{D}'(H)$. We shall only detail the proof of (iii)~$\Rightarrow$~(i).

    Suppose that $\pi_*$ continuously maps $A_c(G)$ in $A(H)$ and let $K$ be a compact subset of $G$. Let $L$ be a compact neighbourhood of $K$ in $G$ and let $\varphi\in\mathscr{D}(G)$ be such that $\varphi=1$ on a neighbourhood of $K$ in $L$ and $\varphi=0$ outside of $L$. There exists a constant $B\geq 0$ such that $\lVert\pi_*(\varphi f)\rVert_{A(H)}\leq B\lVert f\rVert_{A(G)}$ for every $f\in A(G)$. For every $m\in\mathscr{E}'(G)$ supported in $K$ and every $f,g\in\mathscr{D}(H)$, it then holds that
	\begin{align*}
		\langle(\pi_*m)f,g\rangle
		&=\langle\pi_*m,fg\rangle\\
		&=\langle m,\pi^*(fg)\rangle
		=\langle m,\pi^*(f)\pi^*(g)\rangle
		=\langle m,\varphi\pi^*(f)\varphi\pi^*(g)\rangle
	\end{align*}
	and therefore that
	\[
		\lvert\langle(\pi_*m)f,g\rangle\rvert\leq\lVert m\rVert_{\operatorname{M}^{p,q}(G)}\lVert\mathscr{F}_G(\varphi\mkern1.5mu\pi^*f)\rVert_{L^p(\widehat{G{}})}\lVert\mathscr{F}_G(\varphi\mkern1.5mu\pi^*g)\rVert_{L^{q\mathrlap{'}}(\widehat{G{}})}
	\]
	for every $p,q\in[1,\infty]$. Hence, to conclude the proof, it is sufficient to prove the existence of a constant $C\geq 0$ such that $\lVert\mathscr{F}_G(\varphi\mkern1.5mu\pi^*f)\rVert_{L^r(\widehat{G{}})}\leq C^{1/2}\lVert\mathscr{F}_H(f)\rVert_{L^r(\widehat{H{}})}$ for every $f\in\mathscr{D}(H)$ and every $r\in[1,\infty]$. By complex interpolation, it is then sufficient to consider the cases $r=\infty$ and $r=1$ only.
	
	From the functorial properties of the Fourier-Stieltjes algebra with respect to pullbacks (see \cite[2.20, 1\textsuperscript{o}]{Eymard}), one can see that
	\begin{align*}
		\lVert\mathscr{F}_G(\varphi\mkern1.5mu\pi^*f)\rVert_{L^1(\widehat{G{}})}
		&=\lVert\varphi\mkern1.5mu\pi^*f\rVert_{B(G)}\\
		&\leq\lVert\varphi\rVert_{A(G)}\lVert\pi^*f\rVert_{B(G)}
		\leq\lVert\varphi\rVert_{A(G)}\lVert f\rVert_{B(H)}
		=\lVert\varphi\rVert_{A(G)}\lVert\mathscr{F}_H(f)\rVert_{L^1(\widehat{H{}})}
	\end{align*}
	for every $f\in\mathscr{D}(G)$. This proves the case $r=1$. We now treat the case $r=\infty$. Since $\lVert\pi_*(\varphi g)\rVert_{A(G)}\leq B\lVert g\rVert_{A(G)}$ for every $g\in\mathscr{D}(G)$, it follows by transposition that
	\begin{align*}
		\lVert\mathscr{F}_G(\varphi\mkern1.5mu\pi^*f)\rVert_{L^\infty(\widehat{G{}})}
		&=\lVert\varphi\mkern1.5mu\pi^*f\rVert_{A'(G)}
		\leq B\lVert f\rVert_{A'(G)}=B\lVert\mathscr{F}_H(f)\rVert_{L^\infty(\widehat{H{}})}.
	\end{align*}
	for every $f\in\mathscr{D}(H)$, which concludes the proof.
\end{proof}
\begin{remark}\label{remark:estimate}
    Suppose that the conclusion of Jodeit's theorem holds for $\pi$. The proof of \autoref{lemma:pq11} provides a somewhat concrete way to estimate the $\operatorname{M}^{p,q}$ norm of a pushed forward symbol from that of the original symbol: if $K$ is a compact subset of $G$ and $\varphi\in\mathscr{D}(G)$ is such that $\varphi=1$ on a neighbourhood of $K$ in $G$, then the usual bounds obtained by complex interpolation can be rewritten into the relation
    \[
        \lVert\pi_*\rVert_{\operatorname{M}^{p,q}_K(G)\to \operatorname{M}^{p,q}(H)}\leq\lVert\varphi\rVert_{A(G)}^{\frac 1{p\mathrlap{'}}+\frac 1q}\lVert f\mapsto\pi_*(\varphi f)\rVert_{A(G)\to A(H)}^{\frac 1p+\frac 1{q\mathrlap{'}}}.
    \]
\end{remark}
\begin{remark}
    When $G$ is abelian, it is somewhat well known that if $\pi$ is open, then $\pi_*(A_c(G))$ is contained in $A(H)$. Some interesting norm estimates were obtained by A.~Fig\'a-Talamanca and G.~Gaudry (see \cite[cor.~4a of th.~4]{CowlingExtensionLpLq}) but involve constants that are not very concrete.
\end{remark}
\begin{remark}\label{remark:abelianpushAcinAc}
    Here is another argument for why $\pi_*(A_c(G))\subseteq A(H)$ when $G$ is abelian and $\pi$ is open. Every function in $A_c(G)$ is a linear combination of four continuous positive definite functions and for every compact subset $K$ of $G$, there is a linear combination of four compactly supported positive definite continuous functions which is equal to $1$ on $K$ \cite[lemma~3.2]{Eymard}. Since the product of two positive definite continuous functions is positive definite \cite[th.~13]{Godement}, every function in $A_c(G)$ is a linear combination of sixteen compactly supported positive definite continuous functions. Since $\pi$ is open, it already pushes forward compactly supported continuous functions into compactly supported continuous functions \cite[chap.~VII, \textsection\nobreak\,2, n\textsuperscript{o}\nobreak\,2]{BourbakiINTd}. Thanks to Bochner's theorem, it is then easy to check that if $f$ is a compactly supported positive definite continuous function on $G$, then the compactly supported continuous function $\pi_*f$ on $H$ is positive definite and therefore in $A(H)$: indeed the Fourier transform of $\pi_*f$ is none other than the pullback $\hat\pi^*(\hat f)$ of $\hat f$ by the homomorphism $\hat\pi$ dual of $\pi$ and is of course $\geq 0$. We will see later that this argument does not work with general non-abelian groups (\autoref{corollary:pushcspdc} of \autoref{theorem:pushrightstrict}).
\end{remark}
We can now easily obtain the necessity in \autoref{theorem:ncjodeit}.
\begin{proposition}\label{proposition:isopen}
	Let $G$ and $H$ be two locally compact groups and let $\pi$ be a continuous homomorphism from $G$ to $H$ such that the conclusion of Jodeit's theorem holds for $\pi$. Then, $\pi$ is open.
\end{proposition}
\begin{proof}
	Let $V$ be a neighbourhood of $e_G$ in $G$. We will show that $\pi(V)$ is a neighbourhood of $e_H$ in $H$. Let $f$ be a non-zero non-negative smooth function on $G$ whose support is compact and contained in $V$. The set $\pi(V)$ contains the support of the distribution $\pi_*f$. Since $f\in A_c(G)$, the distribution $\pi_*f$ is actually a continuous function (\autoref{lemma:pq11}) and to conclude the proof, it is therefore sufficient to prove that $(\pi_*f)(e_H)\neq 0$. This however follows from
	\[
		(\pi_*f)(e_H)=(\pi^*\pi_*f)(e_G)=(f*(\mu_N)^\flat)(e_G)=\int_Nf\,d\mu_N>0
	\]
	(\autoref{proposition:pushopenD}) since $f$ has to be $>0$ on a neighbourhood of $e_G$ in $N$.
\end{proof}

\subsection{Schur multipliers lemmas}
Our proof that $\pi_*$ maps $A_c(G)$ in $A(H)$ when $\pi$ is open will rely on two important results about completely bounded Schur multipliers of Schatten classes. The first of these results is a transference theorem from Schur multipliers to Fourier multipliers and the second is a pullback theorem of symbols of Schur multipliers. These will merely be variants of \cite[th.~3.1 and lemma~2.1]{ParcetdelaSalleTablate}.
\vskip\baselineskip
Let $X$ and $X'$ be two abstract measure spaces. If $1\leq p<\infty$ (\emph{resp.}\ $p=\infty$), we shall denote by ${S}^p(X,X')$ the space of all bounded linear operators $A$ from $L^2(X)$ to $L^2(X')$ such that $\operatorname{Tr}((A^*A)^{p/2})$ is finite (\emph{resp.}\ the space of all compact linear operators from $L^2(X)$ to $L^2(X')$). ${S}^\infty(X,X')$ will be endowed with the operator space structure obtained by viewing the elements of ${S}^\infty(X,X')$ as operators on $L^2(X)\oplus L^2(X')$ \cite[p.~22]{PisierIntrotoOS}. ${S}^1(X,X')$ will be endowed with the operator space structure \emph{opposite} of that induced by the dual of ${S}^\infty(X',X)$ \cite[n\textsuperscript{o}\nobreak\,2.10]{PisierIntrotoOS} and ${S}^p(X,X')$, if $1<p<\infty$, will be endowed with the operator space structure obtained by the lower complex interpolation method \cite[n\textsuperscript{o}\nobreak\,2.7]{PisierIntrotoOS}.

If $1\leq p\leq 2$, ${S}^p(X,X')$ can be viewed as a subspace of $L^2(X'\times X)$, with equality if $p=2$. If $p>2$ and $X$ and $X'$ are measurable subsets of locally compact groups, ${S}^p(X,X')$ can be viewed as a space of (germs of) distributions thanks to the Schwartz kernel theorem \cite[th.~3]{Bruhat}. The algebraic tensor product $L^2(X')\otimes L^2(X)$ can always be identified with a subspace of ${S}^p(X,X')$. We shall denote by $\operatorname{MS}^p(X,X')$ (\emph{resp.}\ $\operatorname{cbMS}^p(X,X')$) the space of all $m\in L^\infty(X'\times X)$ such that the map $k\mapsto m\cdot k$ from $L^2(X')\otimes L^2(X)$ to ${S}^2(X\times X')$ can be extended into a bounded (\emph{resp.}\ completely bounded) linear operator on ${S}^p(X,X')$ and norm it by the operator norm (\emph{resp.}\ completely bounded operator norm).
\vskip\baselineskip
We first turn to the transference theorem. If $m\in\mathscr{D}'(G)$, we will denote by $\tilde m$ the image of $m$ by the transpose of the continuous linear map from the completed inductive tensor product $\mathscr{D}(G)\overline\otimes\mathscr{D}(G)\simeq\mathscr{D}(G\times G)$ to $\mathscr{D}(G)$ associated to the separately continuous bilinear map $(g,f)\mapsto g*\bar f^\sharp$, so that if $f\in\mathscr{E}(G)$ and $k\in\mathscr{E}(G\times G)$ is the kernel distribution of the operator $g\mapsto f*g$, then $\tilde m\cdot k$ is the kernel distribution of the operator $g\mapsto (mf)*g$. In particular, $\tilde m(s,t)=m(st^{-1})$ when $m$ is a sufficiently regular function. If $X$ is a measurable subset of $G$ endowed with the restriction of the left Haar measure, we will identify $L^2(X)$ with the subspace of $L^2(G)$ consisting of the equivalence classes of square integrable functions on $G$ having a representative vanishing in the complement of $X$ in $G$. In this case, if $X$ and $X'$ are two measurable subsets of $G$, there is a canonical continuous linear map from the completed projective tensor product $L^2(X')\widehat{\otimes} L^2(X)$ to $A(G)$, which is associated to the continuous bilinear map $(g,f)\mapsto g*\bar f^\flat$.

The following statement synthesises the known local and global transference theorems for Schur and Fourier multipliers. Its main improvement over \cite[th.~3.1]{ParcetdelaSalleTablate} is its more precise quantitative nature.
\begin{theorem}\label{theorem:SFtransference}
    Let $X$ and $X'$ be two measurable subsets of $G$. Let $m\in L^\infty(G)$, let $p\in[1,\infty]$ and let $C\geq 0$. Suppose that for the weak topology $\sigma(L^\infty(G),L^1(\supp m))$, the constant function $1$ on $G$ is a limit point of the canonical image in $A(G)$ of
    \[
        \left\{k\in L^2(X')\widehat{\otimes} L^2(X)~\middle\vert~\lVert k\rVert_{L^2(X')\widehat{\otimes} L^2(X)}\leq C\right\}.
    \]
    Then, for $m$ to belong to $\operatorname{M}^{p,p}(G)$ (\emph{resp.}\ $\operatorname{cbM}^{p,p}(G)$), it is sufficient (\emph{resp.}\ necessary and sufficient) that the restriction of $\tilde m$ to $X'\times X$ belongs to $\operatorname{MS}^p(X,X')$ (\emph{resp.}\ $\operatorname{cbMS}^p(X,X')$), in which case
	\[
		\lVert m\rVert_{\operatorname{M}^{p,p}(G)}\leq C\lVert\tilde m\rVert_{\operatorname{MS}^p(X,X')}
	\]
	\[
		\left(\text{resp.}\quad\lVert\tilde m\rVert_{\operatorname{cbMS}^p(G,G)}\leq\lVert m\rVert_{\operatorname{cbM}^{p,p}(G)}\leq C\lVert\tilde m\rVert_{\operatorname{cbMS}^p(X,X')}\right).
	\]
\end{theorem}
\noindent In particular, if $m$ is supported in a non-negligible compact subset $K$ of $G$ and $L$ is another non-negligible compact subset of $G$, then one has
\[
    \lVert m\rVert_{\operatorname{M}^{p,p}(G)}\leq\biggl(\frac{\mu_G(KL)}{\mu_G(L)}\biggr)^{\mkern-5mu 1/2}\lVert\tilde m\rVert_{\operatorname{MS}^p(L,KL)}
\]
by considering the function ${\boldsymbol 1}_{KL}*(\mu_G(L)^{-1}{\boldsymbol 1}_L)^\flat$ which is $=1$ on $K$.
\vskip\baselineskip
The proof of \autoref{theorem:SFtransference} we will present closely follows those of M.~Caspers and M.~de~la~Salle from \cite{CaspersdelaSalle} and of S.~Neuwirth and \'E.~Ricard from \cite[\textsection\nobreak\,2]{NeuwirthRicard}. We will first need the following lemma. It is only a variant of M.~Caspers and M.~de~la~Salle's \cite[th.~5.1]{CaspersdelaSalle} but we shall still provide a complete proof for the convenience of the reader.
\begin{lemma}\label{lemma:SpfromLp}
    Let $a\in L^2(X)$, let $b\in L^2(X')$, let $p\in[1,\infty]$ and let $f\in\mathscr{D}(G)$. Then,
    \[
        \Bigl\lVert(b\otimes a)\lvert b\otimes(\Delta_Ga)\rvert^{-1/p'}\tilde f\Bigr\rVert_{{S}^p(X,X')}\leq\Bigl(\lVert b\rVert_{L^2(X')}\lVert a\rVert_{L^2(X)}\Bigr)^{1/p}\lVert\hat f\rVert_{L^p(\widehat{G{}})}.
    \]
\end{lemma}
\begin{proof}
    If $p=\infty$, this is obvious as $(1\otimes\Delta_G^{-1})\tilde f$ is the kernel distribution of the operator $\varphi\mapsto f*\varphi$. If $p=1$, this follows from the result of M.~Bo\.zejko and G.~Fendler according to which $\lVert\tilde f\rVert_{\operatorname{cbMS}^1(G,G)}=\lVert f\rVert_{\operatorname{cbMA}(G)}$ \cite{BozejkoFendler} since
    \[
        \lVert(b\otimes a)\tilde f\rVert_{{S}^1(X,X')}\leq\lVert b\otimes a\rVert_{{S}^1(X,X')}\lVert\tilde f\rVert_{\operatorname{MS}^1(X,X')}
    \]
    and $\lVert b\otimes a\rVert_{{S}^1(X,X')}=\lVert b\rVert_{L^2(X')}\lVert a\rVert_{L^2(X)}$ and
    \[
        \lVert \tilde f\rVert_{\operatorname{MS}^1(X,X')}\leq\lVert\tilde f\rVert_{\operatorname{cbMS}^1(G,G)}=\lVert f\rVert_{\operatorname{cbMA}(G)}\leq\lVert f\rVert_{A(G)}=\lVert\hat f\rVert_{L^1(\widehat{G{}})}.
    \]
    We will obtain the desired result for the other values of $p$ using an interpolation theorem for analytic families of operators.    
    
    The function $(v,u)\mapsto(v\otimes u)\lvert v\otimes(\Delta_Gu)\rvert^{-1/p'}\tilde f$ from $L^2(X')\times L^2(X)$ to $\mathscr{D}'(G\times G)$ is clearly continuous and the function $T\mapsto\lVert T\rVert_{{S}^p(G,G)}$ from $\mathscr{D}'(G\times G)$ to $[0,+\infty]$ is clearly lower semi-continuous. For the rest of this proof, we can therefore assume that $a$ and $b$ are linear combinations of indicator functions of non-negligible compact subsets of $X$ and $X'$ respectively. In this case, it is easy to check that for every $g\in\mathscr{D}(G)$, the function
    \[
        z\mapsto T_z(\hat g)=(b\otimes a)(\lvert b\rvert\otimes(\Delta_G\lvert a\rvert))^{z-1}\tilde g
    \]
    from $[0,1]+i\mathbf{R}$ to $\mathscr{D}'(G\times G)$ takes all its values in ${S}^1(X,X')$ and, as a function from $[0,1]+i\mathbf{R}$ to ${S}^1(X,X')$, is continuous and bounded in $[0,1]+i\mathbf{R}$ and holomorphic in $\mathopen]0,1\mathclose[+i\mathbf{R}$. The same is true with ${S}^\infty(X,X')+{S}^1(X,X')$ instead of ${S}^1(X',X')$, since it continuously contains ${S}^1(X,X')$. It is then clear that
    \[
        \lVert T_{z}(\hat g)\rVert_{{S}^p(X,X')}\leq(\lVert b\rVert_{L^2(X')}\lVert a\rVert_{L^2(X)})^{\Re z}\lVert\hat g\rVert_{L^{1/\Re z}(\widehat{G{}})}
    \]
    if $\Re z=0$ or $1$. Hence, the operators $T_z$ can be extended into an analytic family of operators on $\mathscr{C}_0(\widehat{G{}})+L^1(\widehat{G{}})$ which satisfies the assumptions of \cite[th.~1]{CwikelJanson}, which concludes the proof.
\end{proof}
The raison d'être of \autoref{lemma:SpfromLp} is the following special formula relating the operators $\hat f\mapsto (b\otimes a)_\lvert b\otimes(\Delta_ G)\rvert^{-1/p'}\tilde f$ and their transposes.
\begin{lemma}\label{lemma:magicformula}
    Let $a\in L^2(G)$, let $b\in L^2(G)$, let $m\in L^\infty(G)$ and let $p\in[1,\infty]$. Denote by $A$ and $B$ the continuous linear operators from $L^p(\widehat{G{}})$ ($\mathscr{C}_0(\widehat{G{}})$ if $p=\infty$) to ${S}^p(G,G)$ and from $L^{p\mathrlap{'}}(\widehat{G{}})$ ($\mathscr{C}_0(\widehat{G{}})$ if $p=1$) to ${S}^{p\mathrlap{'}}(G,G)$ such that
    \[
        A(\hat f)=(b\otimes a)\lvert b\otimes(\Delta_ Ga)\rvert^{-1/p'}\tilde f\quad\text{and}\quad B(\hat f)=\lvert b\otimes a\rvert\lvert b\otimes(\Delta_ Ga)\rvert^{-1/p'}\tilde f
    \]
    for every $f\in\mathscr{D}(G)$. Then, for every $f\in\mathscr{D}(G)$,
    \[
        {}^{t}B\bigl(\tilde m\cdot A(\hat f)\bigr)=\mathscr{F}_G\bigl((b*\bar a^\flat)mf\bigr).
    \]
\end{lemma}
\begin{proof}
    The transpose of $B$ is given by
    \[
        \mathscr{F}_G^{-1}(\,{}^{t}B(k))(t)=\int_{G}\lvert b(s)a(t^{-1}s)\Delta_G(t^{-1}s)\rvert^{-1/p}k(s,t^{-1}s)\,d\mu_G(s)
    \]
    for sufficiently regular kernels $k$. Moreover,
    \[
        \tilde m(s,t^{-1}s)A(\hat f)(s,t^{-1}s)
        =m(t)b(s)a(t^{-1}s)\lvert b(s)a(t^{-1}s)\Delta_G(t^{-1}s)\rvert^{-1/p'}f(t).
    \]
    Hence,
    \[
        \mathscr{F}_G^{-1}(\,{}^{t}B(\tilde m\cdot A(\hat f)))(t)
        =\left(\int_Gb(s)a(t^{-1}s)\,d\mu_G(s)\right)m(t)f(t)
        =((b*\bar a^\flat)mf)(t)
    \]
    as desired.
\end{proof}
\begin{proof}[Proof of \autoref{theorem:SFtransference}]
    The statement for the completely bounded multiplier norm follows from that for ordinary multiplier norms by replacing $G$ with its product with $\operatorname{SU}(2)$ and $m$ with $m\otimes 1$ and from \cite[th.~4.2]{CaspersdelaSalle}; we shall therefore only prove the latter.

    Let $a_0$, $a_1$, …, $a_n\in L^2(X)$ and let $b_0$, $b_1$, …, $b_n\in L^2(X')$. By the estimate given by \autoref{lemma:SpfromLp} and the special formula given by \autoref{lemma:magicformula}, there are linear operators $A_0$, $A_1$, …, $A_n$ from $L^p(\widehat{G{}})$ ($\mathscr{C}_0(\widehat{G{}})$ if $p=\infty$) to ${S}^p(G,G)$ and linear operators $B_0$, $B_1$, …, $B_n$ from $L^{p\mathrlap{'}}(\widehat{G{}})$ ($\mathscr{C}_0(\widehat{G{}})$ if $p=1$) to ${S}^{p\mathrlap{'}}(G,G)$ such that
    \[
        \left\lVert\sum_{0\leq k\leq n}{}^{t}B_k(\tilde m\,A_k(\cdot))\right\rVert_{L^p(\widehat{G{}})\to S^p(G,G)}
        \leq\lVert\tilde m\rVert_{\operatorname{MS}^p(X,X')}\sum_{0\leq k\leq n}\lVert b_k\rVert_{L^2(X')}\lVert a_k\rVert_{L^2(X)}
    \]
    and such that
    \[
        \sum_{0\leq k\leq n}{}^{t}B_k(\tilde m\,A_k(\hat f))=\mathscr{F}_G\Biggl(\Biggl(\sum_{0\leq k\leq n}b_k*\bar a_k^\flat\Biggr)mf\Biggr)
    \]
    for every $f\in\mathscr{D}(G)$.

    Since $mf$ is an integrable function whose support is compact and contained in that of $m$ and since the function $g\mapsto\lVert\hat g\rVert_{L^p(\widehat{G{}})}$ from $L^1(\supp m)$ to $[0,+\infty]$ is clearly lower semi-continuous for the weak topology of $L^1(\supp m)$, one obtains the desired estimate by choosing $a_0$, $a_1$, …, $a_n$ and $b_0$, $b_1$, …, $b_n$ such that $\sum_{0\leq k\leq n}\lVert a_k\rVert_{L^2(X)}\lVert b_k\rVert_{L^2(X')}\leq C$ and such that $\sum_{0\leq k\leq n}b_k*\bar a_k^\flat$ is very close to $1$ for the topology $\sigma(L^\infty(G),L^1(\supp m))$.
\end{proof}
We now turn to the pullback theorem. We recall that if $X$ is a Hausdorff topological space endowed with a Radon measure, $Y$ is a topological space and $\pi$ is a map from $X$ to $Y$, we say that $\pi$ is \emph{Lusin-measurable} if for every compact subset $K$ of $X$ and every $\varepsilon>0$, there exists a compact subset $L$ of $K$ such that $K\setminus L$ has measure $\leq\varepsilon$ and such that the restriction of $\pi$ of $L$ is continuous (\emph{cf.}\ \cite[chap.~I, \textsection\nobreak\,5, def.~9]{SchwartzRadon}). 
\begin{theorem}\label{theorem:schurpullback}
    Let $X$ and $Y$ (\emph{resp.}\ $X'$ and $Y'$) be two Hausdorff topological spaces, each endowed with a Radon measure, and let $\pi$ (\emph{resp.}\ $\pi'$) be a Lusin-measurable map from $X$ to $Y$. Suppose that locally negligible subsets of $Y$ (\emph{resp.}\ $Y'$) are precisely those whose preimages by $\pi$ (\emph{resp.}\ $\pi'$) are locally negligible in $X$ (\emph{resp.}\ $X'$). Let $m\in L^\infty(Y'\times Y)$ and let $p\in[1,\infty]$. For $m\circ(\pi'\times\pi)$ to belong to $\operatorname{cbMS}^p(X,X')$, it is necessary and sufficient that $m$ belongs to $\operatorname{cbMS}^p(Y,Y')$, in which case
    \[
        \lVert m\circ(\pi'\times\pi)\rVert_{\operatorname{cbMS}^p(X',X)}=\lVert m\rVert_{\operatorname{cbMS}^p(Y',Y)}.
    \]
\end{theorem}
When $G$ is a locally compact group, $N$ is a closed normal subgroup of $G$ and $\pi$ is the quotient homomorphism from $G$ onto $G/N$, the locally negligible subsets of $G/N$ are precisely those whose preimages by $\pi$ are locally negligible in $G$ \cite[chap.~VII, \textsection\nobreak\,2, th.~1]{BourbakiINTd}. The above result is therefore to be related to K.~de~Leeuw and S.~Saeki's pullback theorems by quotient homomorphisms \cite[\textsection\nobreak\,3]{Saeki}. More generally, the above pullback theorem combined with the Schur-Fourier transference theorem (\autoref{theorem:SFtransference}) yields the following pullback theorem for completely bounded Fourier multipliers by open homomorphisms on amenable groups:
\begin{corollary}\label{corollary:fourierpullback}
    Let $G$ and $H$ be two locally compact groups and let $\pi$ be an open continuous homomorphism from $G$ to $H$. Suppose that $G$ is amenable, let $m\in L^\infty(H)$ and let $p\in[1,\infty]$. For $m\circ\pi$ to belong to $\operatorname{cbM}^{p,p}(G)$, it is sufficient that $m$ belongs to $\operatorname{cbM}^{p,p}(H)$, in which case
    \[
        \lVert m\circ\pi\rVert_{\operatorname{cbM}^{p,p}(G)}\leq\lVert m\rVert_{\operatorname{cbM}^{p,p}(H)}.
    \]
    If additionally $m$ is supported in the image of $\pi$, then for $m\circ\pi$ to belong to $\operatorname{cbM}^{p,p}(G)$), it is necessary and sufficient that $m$ belongs to $\operatorname{cbM}^{p,p}(H)$, in which case
    \[
        \lVert m\circ\pi\rVert_{\operatorname{cbM}^{p,p}(G)}=\lVert m\rVert_{\operatorname{cbM}^{p,p}(H)}.
    \]
\end{corollary}
Without amenability, we still have the following <<~pullback and truncation~>> theorem:
\begin{corollary}\label{corollary:fourierpullbacktruncate}
    Let $G$ and $H$ be two locally compact groups and let $\pi$ be an open continuous homomorphism from $G$ to $H$. Let $m\in L^\infty(G)$, let $p\in[1,\infty]$, let $\varphi\in\operatorname{M}^{p,p}_c(G)$ (\emph{resp.}\ $\operatorname{cbM}^{p,p}_c(G)$) and denote by $S$ its support. For $\varphi\cdot(m\circ\pi)$ to belong to $\operatorname{M}^{p,p}(G)$ (\emph{resp.}\ $\operatorname{cbM}^{p,p}(G)$, is it sufficient that $m$ belongs to $\operatorname{cbM}^{p,p}(H)$, in which case for every non-negligible compact subset $L$ of $G$,
    \[
        \lVert\varphi\cdot(m\circ\pi)\rVert_{\operatorname{M}^{p,p}(G)}\leq\Bigl(\frac{\mu_G(SL)}{\mu_G(L)}\Bigr)^{1/2}\lVert m\rVert_{\operatorname{cbM}^{p,p}(H)}
    \]
    \[
        \biggl(\text{\emph{resp.}\ }\lVert\varphi\cdot(m\circ\pi)\rVert_{\operatorname{cbM}^{p,p}(G)}\leq\Bigl(\frac{\mu_G(SL)}{\mu_G(L)}\Bigr)^{1/2}\lVert m\rVert_{\operatorname{cbM}^{p,p}(H)}\biggr).
    \]
\end{corollary}
The proof of \autoref{theorem:schurpullback} will take several steps. Below, by a homomorphism of measure spaces from a measure space $X$ to another measure space $Y$, we mean a measurable map $\pi$ from $X$ to $Y$ such that the measure of $Y$ is the pushforward by $\pi$ of the measure of $X$. If $\pi$ is a homomorphism of measure spaces from $X$ to $Y$, the pullback operator $\pi^*:f\mapsto f\circ\pi$ isometrically maps $L^p(Y)$ in $L^p(X)$ for every $p\in[1,\infty]$ and we can therefore define by transposition a pushforward operator $\pi_*$ from $L^q(X)$ to $L^q(Y)$ for every $q\in[1,\infty]$.
\begin{proposition}[\emph{cf.}\ {\cite[lemma~1.13]{LafforguedelaSalle}}]\label{proposition:schurpullback}
    Let $X$ and $Y$ (\emph{resp.}\ $X'$ and $Y'$) be probability spaces and let $\pi$ (\emph{resp.}\ $\pi'$) be a homomorphism of measure spaces from $X$ to $Y$ (\emph{resp.}\ from $X'$ to $Y'$). For $m\circ(\pi'\times\pi)$ to belong to $\operatorname{MS}^p(X,X')$, it is necessary that $m$ belongs to $\operatorname{MS}^p(Y,Y')$ and sufficient that $1\otimes m$ belongs to $\operatorname{MS}^p(X\times Y,X'\times Y')$, in which case
    \[
        \lVert m\rVert_{\operatorname{MS}^p(Y,Y')}\leq\lVert m\circ(\pi'\times\pi)\rVert_{\operatorname{MS}^p(X,X')}\leq\lVert 1\otimes m\rVert_{\operatorname{MS}^p(X\times Y,X'\times Y')}
    \]
    and for $m\circ(\pi'\times\pi)$ to belong to $\operatorname{cbMS}^p(X,X')$, it is necessary and sufficient that $m$ belongs to $\operatorname{cbMS}^p(Y,Y')$, in which case
    \[
        \lVert m\circ(\pi'\times\pi)\rVert_{\operatorname{cbMS}^p(X',X)}=\lVert m\rVert_{\operatorname{cbMS}^p(Y',Y)}.
    \]
\end{proposition}
For the proof of \autoref{proposition:schurpullback}, we will need to establish a few functorial properties of Schatten classes.
\begin{lemma}\label{lemma:Sppushpull}
    Let $X$ and $Y$ (\emph{resp.}\ $X'$ and $Y'$) be probability spaces and let $\pi$ (\emph{resp.}\ $\pi'$) be a homomorphism of measure spaces from $X$ to $Y$ (\emph{resp.}\ from $X'$ to $Y'$). Let $k\in L^2(X')\otimes L^2(X)$, let $l\in L^2(Y')\otimes L^2(Y)$ and let $p\in[1,\infty]$. Then,
    \[
        \lVert(\pi'\times\pi)_*k\rVert_{{S}^p(Y,Y')}\leq\lVert k\rVert_{{S}^p(X,X')}
    \]
    and
    \[
        \lVert(\pi'\times\pi)^*l\rVert_{{S}^p(X,X')}=\lVert l\rVert_{{S}^p(Y,Y')}.
    \]
\end{lemma}
\begin{proof}
  We first prove the first relation. If $f\in L^2(X)$ and $f'\in L^2(X')$, then
    \begin{align*}
        \lVert(\pi'\times\pi)_*(f'\otimes f)\rVert_{{S}^1(Y,Y')}
        &=\lVert\pi'_*(f')\otimes\pi_*(f)\rVert_{{S}^1(Y,Y')}\\
        &=\lVert\pi'_*f'\rVert_{L^2(Y')}\lVert\pi_*f\rVert_{L^2(Y)}\\
        &\leq\lVert f'\rVert_{L^2(X')}\lVert f\rVert_{L^2(X)},
    \end{align*}
    hence $\lVert(\pi'\times\pi)_*k\rVert_{{S}^1(Y,Y')}\leq\lVert k\rVert_{{S}^1(X,X')}$ thanks to the universal property of the projective tensor product. Similarly if $g\in L^2(Y)$ and $g'\in L^2(Y')$, then
    \begin{align*}
        &\left\lvert\iint_{Y'\times Y}g'(y')(\pi'\times\pi)_*k(y',y)g(y)\,dy'\,dy\right\rvert\\
        &=\left\lvert\iint_{X'\times X}g'(\pi'(x'))k(x',x)g(\pi(x))\,dx'\,dx\right\rvert\\
        &\leq\lVert k\rVert_{{S}^\infty(X,X')}\lVert\pi'^*g'\rVert_{L^2(X')}\lVert\pi^*g\rVert_{L^2(X)}\vphantom{\left\lvert\int_{X'}\right.}
        =\lVert k\rVert_{{S}^\infty(X,X')}\lVert g'\rVert_{L^2(Y')}\lVert g\rVert_{L^2(Y)},\vphantom{\left\lvert\int_{X'}\right.}
    \end{align*}
    hence $\lVert(\pi'\times\pi)_*k\rVert_{{S}^\infty(Y,Y')}\leq\lVert k\rVert_{{S}^\infty(X,X')}$. The first relation for general $p\in[1,\infty]$ follows for the case $p=1$ and $p=\infty$ by interpolation.
    
    To prove the second relation, it sufficient to only prove the inequality <<~$\leq$~>> since $l=(\pi'\times\pi)_*(\pi'\times\pi)^*l$ thanks to the first relation. If $g\in L^2(Y)$ and $g'\in L^2(Y')$, then
    \begin{align*}
        \lVert(\pi'\times\pi)^*(g'\otimes g)\rVert_{{S}^1(X,X')}
        &=\lVert\pi'^*g'\otimes\pi^*g\rVert_{{S}^1(X,X')}\\
        &=\lVert\pi'^*g'\rVert_{L^2(X')}\lVert\pi^*g\rVert_{L^2(X)}\\
        &=\lVert g'\rVert_{L^2(Y')}\lVert g\rVert_{L^2(Y)},
    \end{align*}
    hence $\lVert(\pi'\times\pi)^*l\rVert_{{S}^1(X',X)}\leq\lVert l\rVert_{{S}^1(Y,Y)}$ thanks to the universal property of the projective tensor product. Similarly if $f\in L^2(X)$ and $f'\in L^2(X')$, then
    \begin{align*}
        &\left\lvert\iint_{X'\times X}f'(x')(\pi'\times\pi)^*l(x',x)f(x)\,dx'\,dx\right\rvert\\
        &=\left\lvert\iint_{Y'\times Y}\pi'_*f'(y')l(y',y)\pi_*f(y)\,dy'\,dy\right\rvert\\
        &\leq\lVert l\rVert_{{S}^\infty(Y,Y')}\lVert\pi'_*f'\rVert_{L^2(Y')}\lVert\pi_*f\rVert_{L^2(Y)}
        \leq\lVert l\rVert_{{S}^\infty(Y,Y')}\lVert f'\rVert_{L^2(Y')}\lVert f\rVert_{L^2(Y)},\vphantom{\left\lvert\iint_{Y'}\right.}
    \end{align*}
    hence $\lVert(\pi'\times\pi)^*l\rVert_{{S}^\infty(X,X')}\leq\lVert l\rVert_{{S}^\infty(Y,Y')}$. The inequality <<~$\leq$~>> in the second relation for general $p\in[1,\infty]$ then follows from the special cases $p=1$ and $p=\infty$ by interpolation.
\end{proof}
We will also need the following projective limit theorem.
\begin{proposition}[\emph{cf.}\ {\cite[lemma~1.11]{LafforguedelaSalle}}]\label{proposition:schurlimproj}
    Let $I$ (\emph{resp.}\ $I'$) be an upward directed set, let $(X_i,\pi_{ij})$ (\emph{resp.}\ $(X'_{i\smash{'}},\pi_{i\smash{'}j\smash{'}})$) be a projective system of probability spaces relative to $I$ (\emph{resp.}\ $I'$) which admits a projective limit $X$ (\emph{resp.}\ $X'$) and for every $i\in I$ (\emph{resp.}\ $i'\in I'$), denote by $\pi_i$ (\emph{resp.}\ $\pi'_{i\smash{'}}$) the canonical homomorphism from $X$ to $X_i$ (\emph{resp.}\ from $X'$ to $X'_{i\smash{'}}$). Let $m\in L^\infty(X'\times X)$ and let $p\in[1,\infty]$. Then,
    \[
        \lVert m\rVert_{\operatorname{MS}^p(X,X')}=\sup_{(i,i')\in I\times I'}\lVert(\pi'_{i\smash{'}}\times\pi_i)_*m\rVert_{\operatorname{MS}^p(X_i,X'_{i\smash{'}})}
    \]
    and
    \[
        \lVert m\rVert_{\operatorname{cbMS}^p(X,X')}=\sup_{(i,i')\in I\times I'}\lVert(\pi'_{i\smash{'}}\times\pi_i)_*m\rVert_{\operatorname{cbMS}^p(X_i,X'_{i\smash{'}})}.
    \]
\end{proposition}
\begin{proof}
    The result for the completely bounded multiplier norms can be deduced from that for the ordinary multiplier norms by replacing $X$, $X'$ and each $X_i$ and $X_i$ and $X'_{i\smash{'}}$ with their product with an atomic probability spaces $Z$ with infinitely many atoms, each homomorphism with its product with $\operatorname{id}_Z$ and $m$ with $m\otimes 1$ \cite[lemma~1.5 and lemma~1.9]{LafforguedelaSalle}. We shall only detail the proof of the first equality.
    
    Note that for every $(i,i')\in I\times I'$ and every $k,l\in L^2(X'_{i\smash{'}})\otimes L^2(X_i)$, we have
    \begin{multline*}
        \left\lvert\iint_{X'_{i\smash{'}}\times X_i}(\pi'_{i\smash{'}}\times\pi_i)_*m(z',z)k(z',z)l(z',z)\,dz'\,dz\right\rvert\\
        =\left\lvert\iint_{X'\times X}m(x',x)k(\pi'_{i\smash{'}}(x'),\pi_i(x))l(\pi'_{i\smash{'}}(x'),\pi_i(x))\,dx'\,dx\right\rvert.
    \end{multline*}
    The second member of this equality is of course majorised by
    \[
        \lVert m\rVert_{\operatorname{MS}^p(X,X')}\lVert k\circ(\pi'_{i\smash{'}}\times\pi_i)\rVert_{{S}^p(X',X)}\lVert{}^{t}l\circ(\pi_i\times\pi'_{i\smash{'}})\rVert_{{S}^{p\mathrlap{'}}(X',X)}
    \]
    and thanks to the pullback theorem for Schatten classes (\autoref{lemma:Sppushpull}), we then have
    \[
        \lVert m\rVert_{\operatorname{MS}^p(X,X')}\leq\lVert k\rVert_{{S}^p(X_i,X'_{i\smash{'}})}\lVert{}^{t}l\rVert_{{S}^{p\mathrlap{'}}(X'_{i\smash{'}},X_i)}
    \]
    and therefore
    \[
        \lVert(\pi'_{i\smash{'}}\times\pi_i)_*m\rVert_{\operatorname{MS}^p(X_i,X'_{i\smash{'}})}\leq\lVert m\rVert_{\operatorname{MS}^p(X,X')}.
    \]
    Similarly, the first member of this equality is majorised by
    \[
        \lVert(\pi'_{i\smash{'}}\times\pi_i)_*m\rVert_{\operatorname{MS}^p(X_i,X'_{i\smash{'}})}\lVert k\circ(\pi'_{i\smash{'}}\times\pi_i)\rVert_{{S}^p(X,X')}\lVert{}^{t}l\circ(\pi_i\times\pi'_{i\smash{'}})\rVert_{{S}^{p\mathrlap{'}}(X',X)}
    \]
    and since the subspace of $L^2(X)$ (\emph{resp.}\ $L^2(X')$) consisting of the square integrable function which can be written as $f\circ\pi_i$ (\emph{resp.}\ $f'\circ\pi'_{i\smash{'}}$) for some $i\in I$ and $f\in L^2(X_i)$ (\emph{resp.}\ $i'\in I'$ and $f'\in L^2(X'_{i\smash{'}})$) is dense in $L^2(X)$ (\emph{resp.}\ $L^2(X')$), this implies
    \[
        \lVert(\pi'_{i\smash{'}}\times\pi_i)_*m\rVert_{\operatorname{MS}^p(X_i,X'_{i\smash{'}})}\geq\lVert m\rVert_{\operatorname{MS}^p(X,X')}
    \]
    and concludes the proof.
\end{proof}
\begin{proof}[Proof of \autoref{proposition:schurpullback}]
    The result for the completely bounded multiplier norms follows from that for the ordinary multiplier norms by once again considering suitable products. We shall only prove the latter.
	
The pullback theorem for Schatten classes (\autoref{lemma:Sppushpull}) easily yields the first inequality since
    \begin{align*}
        \lVert m\cdot k\rVert_{{S}^p(Y,Y')}
        &=\lVert(\pi'\times\pi)^*(m\cdot k)\rVert_{{S}^p(X,X')}\\
        &=\lVert(\pi'\times\pi)^*(m)\cdot(\pi'\times\pi)^*(k)\rVert_{{S}^p(X,X')}\\
        &\leq\lVert(\pi'\times\pi)^*m\rVert_{\operatorname{MS}^p(X,X')}\lVert(\pi'\times\pi)^*k\rVert_{{S}^p(X,X')}\\
        &=\lVert(\pi'\times\pi)^*m\rVert_{\operatorname{MS}^p(X,X')}\lVert k\rVert_{{S}^p(Y,Y')}
    \end{align*}
    for every $k\in L^2(Y')\otimes L^2(Y)$. To prove the second inequality, first assume that $Y$ and $Y'$ are countable discrete spaces with atomic probability measures, denote by $\Gamma$ and $\Gamma'$ the graphs of $\pi$ and $\pi'$ and endow them with the probability measures puhforwards by $\rho:x\mapsto(x,\pi(x))$ and $\rho':x'\mapsto(x',\pi'(x'))$ of the measures of $X$ and $X'$. The maps $\rho$ and $\rho'$ are isomorphisms of measure spaces and $(\pi'\times\pi)^*m$ is the pullback by $\rho'\times\rho$ of the restriction to $\Gamma'\times\Gamma$ of $1\otimes m$, hence
    \[
        \lVert(\pi'\times\pi)^*m\rVert_{\operatorname{MS}^p(X,X')}=\lVert1\otimes m\rVert_{\operatorname{MS}^p(\Gamma,\Gamma')}.
    \]
    Now since $Y$ and $Y'$ are countable discrete spaces with atomic measures, the underlying sets of $\Gamma$ and $\Gamma'$ are measurable in $X\times Y$ and $X'\times Y'$ and their measures are equivalent to the measures induced by $X\times Y$ and $X'\times Y'$. Since the $\operatorname{MS}^p(\Gamma,\Gamma')$ norm only depends on the equivalence classes of the measures on $\Gamma$ and $\Gamma'$ \cite[lemma~1.9]{LafforguedelaSalle}, we obtain
    \[
        \lVert 1\otimes m\rVert_{\operatorname{MS}^p(\Gamma,\Gamma')}\leq\lVert 1\otimes m\rVert_{\operatorname{MS}^p(X\times Y,X'\otimes Y')}.
    \]
    In general, if $Y$ and $Y'$ are no longer discrete and atomic, there are upward directed sets $I$ and $I'$ and projective systems $(Y_i,\pi_{ij})$ and $(Y'_{i\smash{'}},\pi'_{i\smash{'}j\smash{'}})$ of atomic probability spaces relative to $I$ and $I'$ respectively whose projective limits are $Y$ and $Y'$. For every $i\in I$ (\emph{resp.}\ $i'\in I'$), denote by $\pi_i$ (\emph{resp.}\ $\pi'_{i\smash{'}}$) the canonical homomorphism from $Y$ to $Y_i$ (\emph{resp.}\ from $Y'$ to $Y'_{i\smash{'}}$) and by $\rho_i$ (\emph{resp.}\ $\rho'_{i\smash{'}}$) the map $x\mapsto(x,_pi_i(x))$ (\emph{resp.}\ $x'\mapsto(x',\pi'_{i\smash{'}}(x'))$) from $X$ to $Y_i$ (\emph{resp.}\ from $X'$ to $Y'_{i\smash{'}}$). Since each $Y_i$ and $Y'_{i\smash{'}}$ is atomic, we have
    \begin{multline*}
        \sup_{(i,i')\in I\times I'}\lVert(\pi'\times\pi)^*(\rho'_{i\smash{'}}\times\rho_i)^*(\rho'_{i\smash{'}}\times\rho_i)_*(1\otimes m)\rVert_{\operatorname{MS}^p(X,X')}\\
        \leq\sup_{(i,i')\in I\times I'}\lVert(\rho'_{i\smash{'}}\times\rho_i)_*m\rVert_{\operatorname{MS}^p(X\times Y_i,X'\times Y'_{i\smash{'}})}.
    \end{multline*}
    The second member of this inequality is none other than the $\operatorname{MS}^p(X\times Y,X'\times Y')$ norm of $1\otimes m$ (\autoref{proposition:schurlimproj}). It is then easy to see that as $i$ and $i'$ grow indefinitely in $I$ and $I'$ respectively, we have
    \[
        \lim_{i,i'}(\pi'\times\pi)^*(\rho'_{i\smash{'}}\times\rho_i)^*(\rho'_{i\smash{'}}\times\rho_i)_*m=(\pi'\times\pi)^*m
    \]
    for the ultraweak topology topology $\sigma(L^\infty(X'\times X),L^1(X'\times X))$, topology for which the function $\lVert\cdot\rVert_{\operatorname{MS}^p(X,X')}$ from $L^\infty(X'\times X)$ to $[0,+\infty]$ is lower semi-continuous. Hence,
    \[
        \lVert(\pi'\times\pi)^*m\rVert_{\operatorname{MS}^p(X,X')}
        \leq\sup_{(i,i')\in I\times I'}\lVert(\pi'\times\pi)^*(\rho'_{i\smash{'}}\times\rho_i)^*(\rho'_{i\smash{'}}\times\rho_i)_*m\rVert_{\operatorname{MS}^p(X,X')},
    \]
    which concludes the proof.
\end{proof}
The proof of \autoref{theorem:schurpullback} will now simply consist in reducing the general case to that of \autoref{proposition:schurpullback}. For that we will need two lemmas.
\begin{lemma}\label{lemma:supcompact}
    Let $X$ and $X'$ be Hausdorff topological spaces, each endowed with a Radon measure, and let $\mathfrak{A}$ and $\mathfrak{A}'$ be upward directed sets of measurable subsets of $X$ and $X'$ respectively such that
    \[
        \sup_{A\in\mathfrak{A}}{\boldsymbol 1}_A=1\text{ in }L^\infty(X)\qquad\text{and}\qquad\sup_{A'\in\mathfrak{A}'}{\boldsymbol 1}_{A'}=1\text{ in }L^\infty(X').
    \]
    Let $m\in L^\infty(X'\times X)$ and let $p\in[1,\infty]$. For $m$ to belong to $\operatorname{MS}^p(X,X')$ (\emph{resp.}\ $\operatorname{cbMS}^p(X,X')$), it is necessary and sufficient that $\left\{{\boldsymbol 1}_{A'\times A}m~\middle\vert~(A,A')\in\mathfrak{A}\times\mathfrak{A}'\right\}$ is a bounded subset of $\operatorname{MS}^p(X,X')$ (\emph{resp.}\ $\operatorname{cbMS}^p(X,X')$), in which case
    \[
        \lVert m\rVert_{\operatorname{MS}^p(X,X')}=\sup_{(A,A')\in\mathfrak{A}\times\mathfrak{A}'}\lVert{\boldsymbol 1}_{A'\times A}m\rVert_{\operatorname{MS}^p(X,X')}
    \]
    \[
        \biggl(\text{\emph{resp.}\ }\lVert m\rVert_{\operatorname{cbMS}^p(X,X')}=\sup_{(A,A')\in\mathfrak{A}\times\mathfrak{A}'}\lVert{\boldsymbol 1}_{A'\times A}m\rVert_{\operatorname{cbMS}^p(X,X')}\biggr).
    \]
\end{lemma}
\begin{proof}
	The result for completely bounded multipliers follows from that for ordinary multipliers by replacing $X$ and $X'$ with their products with an infinite discrete space and their measures with their product with the counting measure \cite[lemma~1.5]{LafforguedelaSalle}, while the result for ordinary multipliers follows from the space
	\[
	    \bigcup_{(A,A')\in\mathfrak{A}\times\mathfrak{A}'}{\boldsymbol 1}_{A'\times A}{S}^p(X,X')
	\]
	being dense in ${S}^p(X,X')$.
\end{proof}
\begin{lemma}\label{lemma:pushconc}
	Let $X$ and $Y$ be Hausdorff topological spaces, each endowed with a Radon measure, and let $\pi$ be a Lusin-measurable map from $X$ to $Y$. Suppose that locally negligible subsets of $Y$ are precisely those whose preimages by $\pi$ are locally negligible in $X$ and denote by $\mathfrak{K}$ the set of all compact subsets of $X$ to which the restriction of $\pi$ is continuous. Then,
	\[
	    \sup_{K\in\mathfrak{K}}{\boldsymbol 1}_K=1\text{ in }L^\infty(X)\qquad\text{and}\qquad\sup_{K\in\mathfrak{K}}{\boldsymbol 1}_{\pi(K)}=1\text{ in }L^\infty(Y).
	\]
\end{lemma}
\begin{proof}
    Let $A$ be a measurable subset of $X$ such that for every $K\in\mathfrak{K}$, ${\boldsymbol 1}_A\geq{\boldsymbol 1}_K$ locally almost everywhere in $X$. We will simply show that the complement of $A$ in $X$ is locally negligible in $X$. Let $K$ be a compact subset of $X\setminus A$ and let $\varepsilon>0$. Since $\pi$ is Lusin-measurable, there exists a compact subset $L$ of $K$, belonging to $\mathfrak{K}$ and such that $K\setminus L$ has measure $\leq\varepsilon$. Since $L\subset X\setminus A$ and ${\boldsymbol 1}_A\geq{\boldsymbol 1}_L$ locally almost everywhere, we see that $L$ is negligible and then that $K$ has measure $\leq\varepsilon$. This being true for every $\varepsilon>0$, $K$ is negligible and this shows that every compact subset of $X\setminus A$ is negligible, hence the first part of the lemma.

    Similarly, if $B$ is a measurable subset of $Y$ such that for every $K\in\mathfrak{K}$, ${\boldsymbol 1}_N\geq{\boldsymbol 1}_{\pi(K)}$ locally almost everywhere in $Y$, the assumptions on the measures of $X$ and $Y$ imply that for every $K\in\mathfrak{K}$, ${\boldsymbol 1}_{\pi^{-1}(B)}\geq{\boldsymbol 1}_K$ locally almost everywhere in $X$. The previous reasoning then shows that $X\setminus\pi^{-1}(B)$ is locally negligible in $X$ and then that $Y\setminus B$ is locally negligible in $Y$, hence the second part of the lemma.
\end{proof}
\begin{proof}[Proof of \autoref{theorem:schurpullback}]
    Thanks to \autoref{lemma:supcompact} and \autoref{lemma:pushconc}, we may assume that there exists a compact subset $K$ (\emph{resp.}\ $K'$) of $X$ (\emph{resp.}\ $X'$) to which the restriction of $\pi$ (\emph{resp.}\ $\pi'$) is continuous, on which the measure of $X$ (\emph{resp.}\ $X'$) is supported and such that the measure of $Y$ (\emph{resp.}\ $Y'$) is supported in the compact set $\pi(K)$ (\emph{resp.}\ $\pi'(K')$). In this case, the measure of $Y$ (\emph{resp.}\ $Y'$) is equivalent to the pushforward by $\pi$ (\emph{resp.}\ $\pi'$) of the measure of $X$ (\emph{resp.}\ $X'$). Since the space $\operatorname{cbMS}^p(Y,Y')$ and its norm only depend on the equivalence classes of the measures of $X$ and $X'$ \cite[lemma~1.9]{LafforguedelaSalle}, we may also assume that the measure of $Y$ (\emph{resp.}\ $Y'$) is \emph{equal} to the pushforward by $\pi$ (\emph{resp.}\ $\pi'$) of the measure of $X$ (\emph{resp.}\ $X'$), in which case the result is already given by \autoref{proposition:schurpullback}.
\end{proof}
\subsection{Proof of the general pushforward theorem}
We now have all the tools needed to prove the sufficiency in \autoref{theorem:ncjodeit}. As was announced, we shall actually prove a somewhat concrete quantitative estimate.
\begin{proposition}\label{proposition:openimpliesjodeit}
	Let $G$ and $H$ be two locally compact groups, let $\pi$ be an open continuous homomorphism from $G$ to $H$ and denote by $N$ its kernel. The conclusion of Jodeit's theorem holds for $\pi$. Moreover, for every compact subset $K$ of $G$, every non-negligible compact subsets $L$ and $M$ of $G$ and every $p,q\in[1,\infty]$, it holds that
	\begin{multline*}
	    \lVert\pi_*\rVert_{\operatorname{M}_K^{p,q}(H)\to M^{p,q}(H)}\\
	    \leq\frac{\mu_G(KL)}{\mu_G(L)}\biggl(\frac{\mu_H\bigl(\pi(KLL^{-1}M)\bigr)\mu_N\bigl(MM^{-1}(KLL^{-1})^{-1}KLL^{-1}\bigr)}{\mu_H\bigl(\pi(M)\bigr)}\biggr)^{\mkern-5mu\frac 1{p}+\frac 1{q\mathrlap{'}}}.
	\end{multline*}
\end{proposition}

\begin{lemma}\label{lemma:pushAcinAc}
    Suppose that $\pi$ is open. Let $\varphi\in\mathscr{D}(G)$ and denote by $S$ its support. Then, the map $T\mapsto\pi_*(\varphi T)$ from $\mathscr{D}'(G)$ to $\mathscr{E}'(H)$ continuously maps $A(G)$ in $A(H)$. Moreover, for every non-negligible compact subset $M$ of $G$ and every $f\in A(G)$, it holds that
    \[
        \lVert\pi_*(\varphi f)\rVert_{A(H)}\leq\frac{\mu_H\bigl(\pi(SM)\bigr)\mu_N\bigl(MM^{-1}S^{-1}S\bigr)}{\mu_H\bigl(\pi(M)\bigr)}\lVert\varphi\rVert_{A(G)}\lVert f\rVert_{A(G)}.
    \]
\end{lemma}
Note that this estimate is non-trivial as the image by $\pi$ of a non-negligible subset of $G$ is non-negligible in $H$ \cite[chap.~VII, \textsection\nobreak\,2, th.~1]{BourbakiINTd}.
\begin{proof}
    Since $\pi$ is open, $\pi_*(\varphi f)$ is a continuous function \cite[chap.~VII, \textsection\nobreak\,2, n\textsuperscript{o}\nobreak\,2]{BourbakiINTd} rather than just a distribution so all the computations that will follow are indeed licit. Thanks to \autoref{lemma:M11cisAc}, we already have
    \[
        \lVert\pi_*(\varphi f)\rVert_{A(H)}\leq\biggl(\frac{\mu_H(\pi(SM))}{\mu_H(\pi(M))}\biggr)^{\mkern-5mu\frac 12}\lVert\pi_*(\varphi f)\rVert_{\operatorname{cbM}^{1,1}(H)}.
    \]
    We then apply the transference theorem from Schur to Fourier multipliers (\autoref{theorem:SFtransference}) to obtain
    \[
        \lVert\pi_*(\varphi f)\rVert_{\operatorname{cbM}^{1,1}(H)}\leq\biggl(\frac{\mu_H(\pi(SM))}{\mu_H(\pi(M))}\biggr)^{\mkern-5mu\frac 12}\lVert(\pi_*(\varphi f))^{\sim}\rVert_{\operatorname{cbMS}^1(\pi(M),\pi(SM))}.
    \]  
    The pullback theorem for Schur multipliers (\autoref{theorem:schurpullback}) then yields
    \[
        \lVert(\pi_*(\varphi f))^{\sim}\rVert_{\operatorname{cbMS}^1(\pi(M),\pi(SM))}=\lVert(\pi_*(\varphi f))^{\sim}\circ(\pi\times\pi)\rVert_{\operatorname{cbMS}^1(M,SM)}.
    \]
    Thanks to \autoref{proposition:pushopenD}, we then have
    \[
        (\pi_*(\varphi f))^{\sim}\circ(\pi\times\pi)=(\pi^*\pi_*(\varphi f))^{\sim}=((\varphi f)*(\mu_N)^\flat)^{\sim}
    \]
    and since $\varphi f$ is supported in $S$, we have
    \[
        \bigl((\varphi f)*(\mu_N)^\flat\bigr)^{\sim}(s,t)=\bigl((\varphi f)*({\boldsymbol 1}_{MM^{-1}S^{-1}S}\mu_N)^\flat\bigr)^{\sim}(s,t)
    \]
    for every $(s,t)\in SM\times M$. The transference theorem from Fourier to Schur multipliers (\autoref{theorem:SFtransference}) then implies that
    \begin{multline*}
        \lVert((\varphi f)*({\boldsymbol 1}_{MM^{-1}S^{-1}S}\mu_N)^\flat)^{\sim}\rVert_{\operatorname{cbMS}^1(M,SM)}\\
        \leq\lVert(\varphi f)*({\boldsymbol 1}_{MM^{-1}S^{-1}S}\mu_N)^\flat\rVert_{\operatorname{cbMA}(G)}
        \leq\lVert(\varphi f)*({\boldsymbol 1}_{MM^{-1}S^{-1}S}\mu_N)^\flat\rVert_{A(G)}.
    \end{multline*}
    The estimate provided by \cite[prop.~2.18, 2\textsuperscript{o}]{Eymard} finally yields
    \begin{align*}
        \lVert(\varphi f)*({\boldsymbol 1}_{MM^{-1}S^{-1}S}\mu_N)^\flat\rVert_{A(G)}
        &\leq\mu_N(MM^{-1}S^{-1}S)\lVert\varphi f\rVert_{A(G)}\\
        &\leq\mu_N(MM^{-1}S^{-1}S)\lVert\varphi\rVert_{A(G)}\lVert f\rVert_{A(G)}
    \end{align*}
    as desired.
\end{proof}
\begin{proof}[Proof of \autoref{proposition:openimpliesjodeit}]
    The fact that the conclusion of Jodeit's theorem holds for $\pi$, or equivalently (\autoref{lemma:pq11}) that $\pi_*$ maps $A_c(G)$ in $A(H)$, is immediately given by \autoref{lemma:pushAcinAc}. We shall only detail the proof of the announced estimate.
    
    Let $K$ be a compact subset of $G$ and let $L$ and $M$ be two non-negligible compact subsets of $G$. Recall that the estimate provided by \autoref{remark:estimate} is
    \[
        \lVert\pi_*\rVert_{\operatorname{M}_K^{p,q}(G)\to\operatorname{M}^{p,q}(H)}\leq\lVert\varphi\rVert_{A(G)}^{\frac 1{p\mathrlap{'}}+\frac 1q}\lVert f\mapsto\pi_*(\varphi f)\rVert_{A(G)\to A(H)}^{\frac 1p+\frac 1{q\mathrlap{'}}}
    \]
    whenever $\varphi\in\mathscr{D}(G)$ is $=1$ on a neighbourhood of $K$ in $G$. Combined with the estimate given by \autoref{lemma:pushAcinAc}, this yields
    \[
        \lVert\pi_*\rVert_{\operatorname{M}_K^{p,q}(G)\to\operatorname{M}^{p,q}(H)}\leq\lVert\varphi\rVert_{A(G)}^2\biggl(\frac{\mu_H\bigl(\pi(SM)\bigr)\mu_N\bigl(MM^{-1}S^{-1}S\bigr)}{\mu_H\bigl(\pi(M)\bigr)}\biggr)^{\mkern-5mu\frac 1p+\frac 1{q\mathrlap{'}}}
    \]
    where $S$ is the support of $\varphi$. To obtain the desired estimate, it then sufficient to take $\varphi$ of the form $\psi*(\mu_G(L)^{-1}{\boldsymbol 1}_L)^\flat$ with $\psi\in\mathscr{D}(G)$ with values in $[0,1]$, equal to $1$ on a neighbourhood of $KL$ in $G$ and with support arbitrarily close to $KL$.
\end{proof}
\begin{corollary}\label{corollary:openpushcbMp}
	Let $G$ and $H$ be two locally compact groups, let $\pi$ be an open continuous homomorphism from $G$ to $H$ and let $p\in[1,\infty]$. The pushforward operator $\pi_*$ realises a strict homomorphism of topological vector spaces from $\operatorname{cbM}_c^{p,p}(G)$ onto the closed subspace of $\operatorname{cbM}_c^{p,p}(H)$ consisting of the symbols in $\operatorname{cbM}_c^{p,p}(H)$ supported in the image of $\pi$.
\end{corollary}
\begin{proof}
    Denote by $F$ the closed subspace in question. The fact that $\pi$ continuously maps $\operatorname{cbM}_c^{p,p}(G)$ in $F$ follows from \autoref{proposition:openimpliesjodeit} by replacing $G$ and $H$ with their products with $\operatorname{SU}(2)$ and $\pi$ with $\pi\times\operatorname{id}_{\operatorname{SU}(2)}$. Since $\pi(G)$ is open in $H$, ${\boldsymbol 1}_{\pi(G)}$ is a smooth positive definite function on $H$ and the projector $m\mapsto{\boldsymbol 1}_{\pi(G)}m$ on $\operatorname{cbM}_c^{p,p}(H)$ is continuous and has image $F$. Hence, $F$ is topological direct factor of $\operatorname{cbM}_c^{p,p}(F)$ and is therefore bornological. To conclude the proof, it is therefore sufficient to prove that every bounded subset of $F$ is contained in the image by $\pi_*$ of a bounded subset of $\operatorname{cbM}_c^{p,p}(G)$.
    
    Let $B$ be a bounded subset of $F$. There is a compact subset of $H$ containing the support of every element of $B$. Since every element of $B$ is supported in $\pi(G)$, there is even a compact subset $K$ of $\pi(G)$ containing the support of every element of $B$. Let $\varphi\in\mathscr{D}(G)$ be such that $\pi_*\varphi=1$ on a neighbourhood of $K$ in $\pi(G)$. By the pullback and truncation theorem (\autoref{corollary:fourierpullbacktruncate} of \autoref{theorem:schurpullback}), the map $\sigma:m\mapsto \varphi\mkern1.5mu\pi^*m$ from $\mathscr{E}'(H)$ to $\mathscr{E}'(G)$ continuously maps $\operatorname{cbM}_c^{p,p}(H)$ in $\operatorname{cbM}_c^{p,p}(G)$. It is then obvious that $\pi_*(\sigma(B))=B$.
\end{proof}
\begin{corollary}
    Let $G$ be an amenable locally compact group, let $N$ be a closed normal subgroup of $G$. Let $m\in\mathscr{E}'(G)$ and let $p\in[1,\infty]$. For $m*(\mu_N)^\flat$ to belong to $\operatorname{cbM}^{p,p}(G)$, it is sufficient that $m$ belongs to $\operatorname{cbM}_c^{p,p}(G)$.
\end{corollary}
\begin{proof}
    Since $m*(\mu_N)^\flat=\pi^*\pi_*m$ if $\pi$ is the canonical homomorphism from $G$ to $G/N$ (\autoref{proposition:pushopenD}), this is an immediate consequence of that facts that $\pi_*$ maps $\operatorname{cbM}^{p,p}_c(G)$ in $\operatorname{cbM}^{p}(G/N)$ (\autoref{corollary:openpushcbMp} of \autoref{proposition:openimpliesjodeit}) and that $\pi^*$ maps $\operatorname{cbM}^p(G/N)$ in $\operatorname{cbM}^{p,p}(G)$ (\autoref{corollary:fourierpullback} of \autoref{theorem:schurpullback}).
\end{proof}
\begin{corollary}
    Let $G$ and $H$ be two locally compact groups and let $\pi$ be an open continuous homomorphism from $G$ to $H$. Suppose that $G$ is abelian and let $p\in[1,\infty]$. The pushforward operator $\pi_*$ realises a strict homomorphism of topological vector spaces from $\operatorname{M}_c^{p,p}(G)$ onto the closed subspace of $\operatorname{M}_c^{p,p}(H)$ consisting of the symbols in $\operatorname{M}_c^{p,p}(H)$ supported in the image of $\pi$.
\end{corollary}
This can be proven in the same fashion as the \autoref{corollary:openpushcbMp}, except the pullback and truncation theorem for completely bounded Fourier multipliers (\autoref{corollary:fourierpullbacktruncate} of \autoref{theorem:schurpullback}) should be replaced with the general pullback theorem for Fourier multipliers on abelian groups (\autoref{theorem:lohoué}).
\vskip\baselineskip
Without the commutativity condition, it can happen that there is a symbol $m\in\operatorname{M}_c^{p,p}(H)$ supported in the image of $\pi$ but not belonging to $\pi_*(\operatorname{M}_c^{p,p}(G))$. This can be shown in the same way as the impossibility of a general pullback theorem (\emph{cf.}\ \cite[\textsection\nobreak\,7]{CaspersParcetPerrinRicard}). 
\begin{proposition}
    Let $H$ be a locally compact group such that there exists a symbol in $\operatorname{M}^{p,p}_c(H)$ that is not in $\operatorname{cbM}^{p,p}_c(H)$. Denote by $G$ the product group $H\times\operatorname{SU}(2)$ and by $\pi$ the canonical projection of $G$ onto $H$. Then, there exists a symbol in $\operatorname{M}^{p,p}_c(H)$ which is not the pushforward by $\pi$ of any symbol in $\operatorname{M}^{p,p}_c(G)$.
\end{proposition}
An example of a locally compact group $H$ such that $\operatorname{M}^{p,p}_c(H)\setminus\operatorname{cbM}^{p,p}_c(H)$ is non-empty for every $p\in[1,\infty]\setminus\{1,2,\infty\}$ is the circle group $\mathbf{T}$ \cite[prop.~8.1.3]{PisierNCvvLp} or more generally any infinite locally compact abelian group \cite[th.~1.5]{ArhancetUnconditionality}.
\begin{proof}
    We will simply show that $\pi_*(\operatorname{M}^{p,p}_c(G))\subseteq\operatorname{cbM}^{p,p}(H)$. Let $m\in\operatorname{M}_c^{p,p}(G)$. By \autoref{proposition:pushopenD}, we have $m*(\mu_{\operatorname{SU}(2)})^\flat=\pi^*\pi_*m=(\pi_*m)\otimes 1$. Hence, in order to conclude that $\pi_*m\in\operatorname{cbM}^{p,p}(H)$, it is sufficient to prove that $m*(\mu_{\operatorname{SU}(2)})^\flat\in\operatorname{M}^{p,p}(G)$. Note that for every $f,g\in\mathscr{D}(G)$, we have
    \begin{multline*}
        \lvert\langle (m*(\mu_{\operatorname{SU}(2)})^\flat)f,g\rangle\rvert\\
        =\left\lvert\int_{\operatorname{SU}(2)}\langle(m*\delta_s^\flat)f,g\rangle\,d\mu_{\operatorname{SU}(2)}(s)\right\rvert
        \leq\int_{\operatorname{SU}(2)}\lvert\langle(m*\delta_s^\flat)f,g\rangle\rvert\,d\mu_{\operatorname{SU}(2)}(s)
    \end{multline*}
    and that for every $f,g\in\mathscr{D}(G)$ and every $s\in\operatorname{SU}(2)$,
    \begin{align*}
        \lvert\langle(m*\delta_s^\flat)f,g\rangle\rvert
        &=\lvert\langle m,(fg)*\delta_s\rangle\rvert\\
        &=\lvert\langle m,(f*\delta_s)(g*\delta_s)\rangle\rvert\vphantom{\widehat{f*\delta_s}}
        \leq\lVert m\rVert_{\operatorname{M}^{p,p}(G)}\lVert\widehat{f*\delta_s}\rVert_{L^p(\widehat{G{}})}\lVert\widehat{g*\delta_s}\rVert_{L^{p\mathrlap{'}}(\widehat{G{}})}.
    \end{align*}
    It is then easy to check that $\rVert\widehat{f*\delta_s}\lVert_{L^r(\widehat{G{}})}\leq\lVert\hat f\rVert_{L^r(\widehat{G{}})}$ for every $f\in\mathscr{D}(G)$, every $s\in\operatorname{SU}(2)$ and every $r\in[1,\infty]$. If $r=\infty$, this is obvious. If $r=1$, this follows from $\Delta_G$ being equal to $1$ on the closed normal subgroup $\{e_H\}\times\operatorname{SU}(2)$ of $G$ \cite[chap.~VII, \textsection\nobreak\,2, prop.~10]{BourbakiINTd} and from \cite[prop.~2.18, 2\textsuperscript{o}]{Eymard}. The general case then follows from the special cases $r=\infty$ and $r=1$ by interpolation. Hence,
    \[
        \lvert\langle (m*(\mu_{\operatorname{SU}(2)})^\flat)f,g\rangle\rvert\leq\lVert m\rVert_{\operatorname{M}^{p,p}(G)}\lVert\hat f\rVert_{L^p(\widehat{G{}})}\lVert\hat g\rVert_{L^{p\mathrlap{'}}(\widehat{G{}})}
    \]
    for every $f,g\in\mathscr{D}(G)$, which shows that $m*(\mu_{\operatorname{SU}(2)})^\flat\in\operatorname{M}^{p,p}(G)$ and concludes the proof.
\end{proof}
\section{Functorial properties of positive definite distributions}\label{section:functorialprop}
In this section, we develop the \autoref{remark:abelianpushAcinAc} following \autoref{lemma:pq11} and study the functorial properties of positive definite distributions in order to characterise the continuous homomorphisms of locally compact groups which push forward every compactly supported positive definite continuous function into a continuous positive definite function.
\subsection{Positive definite distributions and amenability}
We shall say that a distribution $T$ on a locally compact group $G$ is \emph{left positive definite} (\emph{resp.}\ \emph{right positive definite}) if
\[
	\langle T*f|f\rangle\geq 0\qquad(\text{\emph{resp.}\ }\langle f*T|f\rangle\geq 0)
\]
for every $f\in\mathscr{D}(G)$. We therefore deviate from the terminology of the classical texts \cite{DixmierCetoile, Godement, Pier} where positive definite continuous functions and positive definite measures are our right positive definite continuous functions and left positive definite measures respectively. Left and right positive definiteness coincide when $G$ is unimodular: this is because
\[
	\langle T*f|f\rangle=\langle T|f*f^\flat\rangle=\langle f^\sharp*T|f^\flat\rangle
	\quad\text{and}\quad
	\langle f*T|f\rangle=\langle T|f^\sharp*f\rangle=\langle T*f^\flat|f^\sharp\rangle
\]
and the involutions $f\mapsto f^\flat$ and $f\mapsto f^\sharp$ are equal precisely when $\Delta_G=1$. More generally,
\begin{lemma}\label{lemma:leftrightposdef}
	For a distribution $T$ on a locally compact group $G$ to be left (\emph{resp.}\ right) positive definite, it is necessary and sufficient that $\Delta_{\smash{G}}^{1/2}T$ (\emph{resp.}\ $\Delta_{\smash{G}}^{-1/2}T$) is right (\emph{resp.}\ left) positive definite.
\end{lemma}
\begin{proof}
	This follows easily from the relation
	\[
	    \Delta_G^{1/2}(f*f^\sharp)=(\Delta_G^{1/2}f)*(\Delta_G^{1/2}f)^\flat
	\]
	for every $f\in\mathscr{D}(G)$.
\end{proof}
Similarly,
\begin{lemma}
	Let $T$ be a distribution on a locally compact group $G$. For the three distributions $T$, $\overline{T}$ and $T^\sharp$ (\emph{resp.}\ $T^\flat$) to be all left (\emph{resp.}\ right) positive definite, it is necessary and sufficient that at least one of them is, in which case $T=T^\sharp$ (\emph{resp.}\ $T=T^\flat$).
\end{lemma}
\begin{proof}
	This follows easily from the relations
    \[
        \overline{f*f^\flat}=\bar f*\bar f^\flat,
        \quad\overline{f^\sharp*f}=\bar f^\sharp*\bar f,
        \quad(f*f^\flat)^\flat=f*f^\flat
        \quad\text{and}\quad(f^\sharp*f)^\sharp=f^\sharp*f
    \]
	for every $f\in\mathscr{D}(G)$.
\end{proof}
It can be checked that the left (\emph{resp.}\ right) positive definite distributions on $G$ form a \emph{closed} convex pointed salient cone in $\mathscr{D}'(G)$. Similarly, the left (\emph{resp.}\ right) positive definite \emph{compactly supported} distributions on $G$ form a closed convex pointed salient cone in $\mathscr{E}'(G)$.
\begin{lemma}[positive definite biregularisations]\label{lemma:biregularisation}
	Let $T$ be a left (\emph{resp.}\ right) positive definite distribution on a locally compact group $G$. For every $f\in\mathscr{D}(G)$, the smooth function $f^\sharp*T*f$ (\emph{resp.}\ $f*T*f^\flat$) is left (\emph{resp.}\ right) positive definite and is compactly supported whenever $T$ is.
	
	Moreover, if $\mathfrak{F}$ is a filter basis on $\mathscr{D}(G)$ which converges in $\mathscr{E}'(G)$ to $\delta_{e_G}$ and contains a set that is bounded in $\mathscr{E}'(G)$, then
	\[
		\lim_{f,\mathfrak{F}}(f^\sharp*T*f)=T\qquad\left(\text{\emph{resp.}\ }\lim_{f,\mathfrak{F}}(f*T*f^\flat)=T\right)
	\]
	in $\mathscr{D}'(G)$, and also in $\mathscr{E}'(G)$ whenever $T\in\mathscr{E}'(G)$.
\end{lemma}
Such filter bases of course do exist: one may take as usual $\mathfrak{F}$ to consist of the sets
    \[
        \left\{f\in\mathscr{D}(G)~\middle\vert~f\geq 0,~\supp f\subseteq V\text{ and}\int_Gf(s)\,ds=1\right\}
    \]
    as $V$ ranges over a basis of the filter of neighbourhoods of $e_G$ in $G$. In this case, there are sets in $\mathfrak{F}$ that are even bounded in the space $\mathscr{C}'(G)$ of compactly supported complex Radon measures on $G$, strong dual of the space $\mathscr{C}(G)$ of all complex continuous functions on $G$ endowed with the topology of uniform convergence on compact subsets of $G$.
\begin{proof}
	The first part of the lemma follows from the elementary relations
	\[
		\langle f^\sharp*T*f*g|g\rangle=\langle T*f*g|f*g\rangle\quad(\text{\emph{resp.}\ }\langle g*f*T*f^\flat|g\rangle=\langle g*f*T|g*f\rangle)
	\]
	for every $f,g\in\mathscr{D}(G)$. To prove the second part, it is sufficient to show that convolution, as a bilinear map from $\mathscr{E}'(G)\times\mathscr{D}'(G)$ or $\mathscr{D}'(G)\times\mathscr{E}'(G)$ to $\mathscr{D}'(G)$, is hypocontinuous with respect to the bounded subsets of $\mathscr{E}'(G)$. This however is obvious since it is separately continuous \cite[\textsection\nobreak\,6]{Bruhat} and $\mathscr{D}'(G)$ is barrelled \cite[cor.~2 of th.~2]{Bruhat}.
\end{proof}
The various conditions introduced by R.~Godement involving positive definite functions and measures \cite[problems 4 and 5]{Godement}, which are classically known to characterise amenability, have natural analogues in the language of smooth functions and distributions.
\begin{theorem}\label{theorem:godement}
	Let $G$ be a locally compact group. The following conditions are equivalent:
	\begin{enumerate}
		\item $G$ is amenable.
		\item For every left positive definite $f\in\mathscr{D}(G)$ and every right positive definite $T\in\mathscr{D}'(G)$, $\langle f|T\rangle\geq 0$.
		\item For every left positive definite $f\in\mathscr{D}(G)$, $\langle f|1\rangle\geq 0$.
		\item Every right positive definite $T\in\mathscr{D}'(G)$ is a limit point in $\mathscr{D}'(G)$ of the convex cone generated by the distributions $S*{S}^\flat$ with $S\in\mathscr{E}'(G)$.
		\item The constant distribution $1$ is a limit point in $\mathscr{D}'(G)$ of the convex cone generated by the distributions $S*{S}^\flat$ with $S\in\mathscr{E}'(G)$.
		\item For every left positive definite $T\in\mathscr{E}'(G)$ and every right positive definite $f\in\mathscr{E}(G)$, $\langle T|f\rangle\geq 0$.
		\item For every left positive definite $T\in\mathscr{E}'(G)$, $\langle T|1\rangle\geq 0$.
		\item Every right positive definite $f\in\mathscr{E}(G)$ is a limit point in $\mathscr{E}(G)$ of the convex cone generated by the functions $g*g^\flat$ with $g\in\mathscr{D}(G)$.
		\item The constant function $1$ is a limit point in $\mathscr{E}(G)$ of the convex cone generated by the functions $g*g^\flat$ with $g\in\mathscr{D}(G)$.
	\end{enumerate}
\end{theorem}
\begin{proof}[Proof sketch]
	All the implications
	\[\begin{tikzcd}[column sep=small, row sep = small]
		\text{(vii)}\arrow[r, Rightarrow]&\text{(ix)}\arrow[d, Rightarrow]&\arrow[l, Rightarrow]\text{(viii)}\\
		\text{(v)}\arrow[u, Rightarrow]&\text{(i)}\arrow[d, Rightarrow]&\text{(vi)}\arrow[u, Rightarrow]\\
		\text{(iii)}\arrow[u, Rightarrow]&\text{(ii)}\arrow[l, Rightarrow]\arrow[r, Rightarrow]&\text{(iv)}\arrow[u, Rightarrow]
	\end{tikzcd}\]
	can be proven by imitating the proof of \cite[prop.~18.3.6]{DixmierCetoile} or are elementary.
	
	(i) $\Rightarrow$ (ii). Suppose that $G$ is amenable, let $f\in\mathscr{D}(G)$ be left positive definite and let $T\in\mathscr{D}'(G)$ be right positive definite. For every $g\in\mathscr{D}(G)$, the smooth function $g*T*g^\flat$ is right positive definite, hence $\langle f|g*T*g^\flat\rangle\geq 0$ since $G$ is amenable \cite[th.~8.9]{Pier} and then $\langle f|T\rangle\geq 0$ by letting $g$ converge to $\delta_{e_G}$ in $\mathscr{E}'(G)$ while remaining in a bounded subset of $\mathscr{E}'(G)$ (\autoref{lemma:biregularisation}).

	(ii) $\Rightarrow$ (iii). This implication is obvious.
	
	(iii) $\Rightarrow$ (v). Condition (iii) implies that the constant distribution $1$ belongs to the bipolar in $\mathscr{D}'(G)$ of the convex cone generated by the distributions $S*{S}^\flat$ with $S\in\mathscr{E}'(G)$ (see for example the proof of \cite[prop.~18.3.6]{DixmierCetoile} or that of \cite[th.~8.9]{Pier}) and therefore implies condition (v) by the bipolar theorem.
	
	(v) $\Rightarrow$ (vii). Similar biregularisation argument to (i) $\Rightarrow$ (ii).
	
	(vii) $\Rightarrow$ (ix). Similar bipolarity argument to (iii) $\Rightarrow$ (v).
	
	(ix) $\Rightarrow$ (i). This implication follows from the topology of $\mathscr{E}(G)$ being finer than that of uniform convergence on compact sets, since $G$ being amenable is equivalent to the constant function $1$ being a limit point in $\mathscr{C}(G)$ of the convex cone of right positive definite functions in $A(G)$ (see \cite[prop.~18.3.6]{DixmierCetoile} or \cite[prop.~8.8]{Pier}).
	
	(ii) $\Rightarrow$ (iv). Similar bipolarity argument to (iii) $\Rightarrow$ (v). 
	
	(iv) $\Rightarrow$ (vi). Similar biregularisation argument to (i) $\Rightarrow$ (ii).
	
	(vi) $\Rightarrow$ (viii). Similar bipolarity argument to (iii) $\Rightarrow$ (v).
	
	(viii) $\Rightarrow$ (ix). This implication is obvious.
\end{proof}
\begin{remark}
	The above theorem is quite interesting since at first glance, conditions (v) and (ix) seem respectively much weaker and much stronger than the usual uniform approximation on compact sets condition.
\end{remark}
\subsection{Pullbacks of positive definite distributions}
It is well known that pullbacks of right positive definite continuous functions are again right positive definite (see \cite[th.~2.20]{Eymard}). In the case of open homomorphisms, this result can be extended by biregularisation to right positive definite distributions (\autoref{lemma:biregularisation}).%
\begin{proposition}\label{proposition:pullright}
	Let $G$ and $H$ be two locally compact groups and let $\pi$ be a continuous homomorphism from $G$ to $H$. The pullback by $\pi$ of any right positive definite continuous function on $H$ is a right positive definite continuous function on $G$. If additionally $\pi$ is open, then the pullback by $\pi$ of any right positive definite distribution on $H$ is a right positive definite distribution on $G$.
\end{proposition}
Pullbacks of left positive definite continuous functions however may not be left positive definite.
\begin{proposition}\label{proposition:pullleft}
	Let $G$ and $H$ be two locally compact groups and let $\pi$ be a continuous homomorphism from $G$ to $H$. The following conditions are equivalent:
	\begin{enumerate}
		\item The pullback by $\pi$ of any left positive definite continuous function on $H$ is a left positive definite continuous function on $G$.
		\item The pullback by $\pi$ of any left positive definite compactly supported smooth function on $H$ is a left positive definite smooth function on $G$.
		\item The modular function of $G$ is the pullback by $\pi$ of the modular function of $H$.
	\end{enumerate}
	If additionally $\pi$ is open, then the above three conditions are equivalent to the following fourth:
	\begin{enumerate}\setcounter{enumi}{3}
		\item The pullback by $\pi$ of any left positive definite distribution on $H$ is a left positive definite distribution on $G$.
	\end{enumerate}
\end{proposition}
\begin{proof}
	The implication (i) $\Rightarrow$ (ii) is obvious. Because of the relation between left and right positive definiteness (\autoref{lemma:leftrightposdef}), the implication (iii) $\Rightarrow$ (i) follows easily from the analogue statement for right positive definite distributions (\autoref{proposition:pullright}). The equivalence (i) $\Leftrightarrow$ (iv) if $\pi$ is open can easily be obtained by biregularisation (\autoref{lemma:biregularisation}). We shall only detail the proof of (ii) $\Rightarrow$ (iii).
	
	Suppose that $\pi^*(f)$ is left positive definite for every left positive definite $f\in\mathscr{D}(G)$. If $f\in\mathscr{D}(H)$ is left positive definite, then since $\pi^*(f)^\sharp=\pi^*(f)$ and $\pi^*(\Delta_H^{1/2}f)^\flat=\pi^*(\Delta_H^{1/2}f)$, it follows that
	\begin{align*}
		\Delta_G\pi^*(f)&=\Delta_G\pi^*(f)^\sharp=\pi^*(f)^\flat\\
		&=\pi^*(\Delta_H^{1/2})\pi^*(\Delta_H^{1/2}f)^\flat=\pi^*(\Delta_H^{1/2})\pi^*(\Delta_H^{1/2}f)=\pi^*(\Delta_H)\pi^*(f),
	\end{align*}
	hence $\Delta_G=\pi^*(\Delta_H)$ on the subset of $G$ where $\pi^*(f)\neq 0$. We conclude by choosing $f$ of the form $g^\sharp*g$ with $g\in\mathscr{D}(G)$ and $g=1$ on arbitrarily large compact subsets of $H$. 
\end{proof}
\subsection{Pushforwards of positive definite distributions}
In this subsection, we characterise the continuous homomorphisms by which the pushforward of any compactly supported right positive definite contiuous function is again a right positive definite continuous function.
\begin{lemma}\label{lemma:pushleftproper}
	Let $G$ and $H$ be two locally compact groups and let $\pi$ be a \emph{proper} continuous homomorphisms from $G$ to $H$. The pushforward by $\pi$ of any left positive definite compactly supported distribution on $G$ is a left positive definite distribution on $H$.
\end{lemma}
\begin{proof}
	By biregularisation (\autoref{lemma:biregularisation}), it is sufficient to consider compactly supported smooth functions only rather than distributions. Let $f$ be a left positive definite compactly supported smooth function on $G$. For every $g\in\mathscr{D}(H)$, the function $\pi^*(g*g^\flat)$ is a right positive definite compactly supported smooth function on $G$ (\autoref{proposition:pullright}) and the quantity
	\[
		\langle (\pi_*f)*g|g\rangle=\langle\pi_*f|g*g^\flat\rangle=\langle f|\pi^*(g*g^\flat)\rangle
	\]
	is non-negative since, by a result of R.~Godement, the integral of the product of a left positive definite square integrable function with a right positive definite square integrable function is non-negative \cite[prop.~18]{Godement}.
\end{proof}
\begin{theorem}\label{theorem:pushleftstrict}
	Let $G$ and $H$ be two locally compact groups and let $\pi$ be a strict homomorphism of topological groups from $G$ to $H$. The following conditions are equivalent:
	\begin{enumerate}
		\item The pushforward by $\pi$ of any left positive definite compactly supported smooth function on $G$ is a left positive definite distribution on $H$.
		\item The pushforward by $\pi$ of any left positive definite compactly supported distribution on $G$ is a left positive definite distribution on $H$.
		\item The kernel of $\pi$ is amenable.
	\end{enumerate}
\end{theorem}
\begin{proof}
	The implication (i) $\Rightarrow$ (ii) can easily be obtained by biregularisation (\autoref{lemma:biregularisation}). We shall only detail the proofs of (ii) $\Rightarrow$ (iii) and (iii) $\Rightarrow$ (i).
	
	(ii) $\Rightarrow$ (iii). Suppose that $\pi_*T$ is left positive definite for every left positive definite $T\in\mathscr{E}'(G)$ and denote by $N$ the kernel of $\pi$. Let $T$ be a compactly supported left positive definite \emph{Radon measure} on $N$. To prove that $N$ is amenable, it is sufficient to prove that $\langle T|1\rangle\geq 0$ (\autoref{theorem:godement}). Since the canonical injection from $N$ to $G$ is proper, $T$ is still left positive definite when naïvely viewed as a compactly supported measure on $G$ supported in $N$ (\autoref{lemma:pushleftproper}). Now if $g\in\mathscr{D}(H)$ is such that $(g*g^\flat)(e_H)=1$, then $\pi^*(g*g^\flat)=1$ on $N$ and since $\pi_*T$ is left positive definite by assumption, it holds that
	\[
		\langle T|1\rangle=\langle T|\pi^*(g*g^\flat)\rangle=\langle\pi_*T|g*g^\flat\rangle=\langle(\pi_*T)*g|g\rangle\geq 0
	\]
	as desired.
	
	(iii) $\Rightarrow$ (i). Suppose that the kernel $N$ of $\pi$ is amenable and let $f$ be a left positive definite compactly supported smooth function on $G$. Since the canonical injection of $\pi(G)$ into $H$ is proper, we can assume that $\pi$ is surjective to show that $\pi_*f$ is left positive definite (\autoref{lemma:pushleftproper}), in which case $\pi_*$ realises a surjection from $\mathscr{D}(G)$ to $\mathscr{D}(H)$ (\autoref{proposition:pushopenD}).
	
	For every $h\in\mathscr{D}(G)$, it holds that
	\begin{align*}
		\langle(\pi_*f)*(\pi_*h)|\pi_*h\rangle
		&=\langle\pi_*(f*h)|\pi_*h\rangle\\
		&=\langle f*h|\pi^*(\pi_*h)\rangle
		=\langle f*h|h*(\mu_N)^\flat\rangle
		=\langle h^\sharp*f*h|(\mu_N)^\flat\rangle.	
	\end{align*}
	Since $N$ is a closed \emph{normal} subgroup of $G$, its modular function is none other than the restriction of that of $G$ \cite[chap.~VII, \textsection\nobreak\,2, prop.~10 b)]{BourbakiINTd} and since $h^\sharp*f*h$ is left positive definite, its restriction to $N$ is also left positive definite (\autoref{proposition:pullleft}). Since $N$ is amenable, the integral $\langle h^\sharp*f*h|(\mu_N)^\flat\rangle$ is $\geq 0$ (\autoref{theorem:godement}).
\end{proof}
\begin{remark}
    This theorem shows that contrary to the abelian setting, a continuous homomorphism from $G$ to $H$ does not necessarily induce a C\textsuperscript{\rlap*}-algebra homomorphism from $\mathscr{C}_0(\widehat{G{}})$ to $L^\infty(\widehat{H{}})$.
\end{remark}
\begin{theorem}\label{theorem:pushrightstrict}
	Let $G$ and $H$ be two locally compact groups and let $\pi$ be a strict homomorphism of topological groups from $G$ to $H$. The following conditions are equivalent:
	\begin{enumerate}
		\item The pushforward by $\pi$ of any right positive definite compactly supported smooth function on $G$ is a right positive definite distribution on $H$.
		\item The pushforward by $\pi$ of any right positive definite compactly supported distribution on $G$ is a right positive definite distribution on $H$.
		\item The kernel of $\pi$ is amenable and the modular function of $G$ is the pullback by $\pi$ of the modular function of $H$.
	\end{enumerate}
\end{theorem}
\begin{proof}
	The implication (i) $\Rightarrow$ (ii) can easily be obtained by biregularisation (\autoref{lemma:biregularisation}). We shall only detail the proofs of (ii) $\Rightarrow$ (iii) and (iii) $\Rightarrow$ (i).
	
	(ii) $\Rightarrow$ (iii) Suppose that $\pi_*T$ is right positive definite for every right positive definite $T\in\mathscr{E}'(G)$. Since
	\[
	    \Delta_G^{-1}=\bigl({\boldsymbol 1}_G\bigr)^\sharp=\bigl(\pi^*{\boldsymbol 1}_H\bigr)^\sharp=\bigr(\pi^*\bigl((\Delta_H^{-1})^\sharp\bigr)\bigl)^\sharp,
	\]
	it holds that
	\begin{align*}
		\langle T|\Delta_G^{-1}\rangle
		&=\langle T|\pi^*((\Delta_H^{-1})^\sharp)^\sharp\rangle\\
		&=\langle\pi_*(T^\flat)^\flat|\Delta_H^{-1}\rangle=\langle\pi_*T|\Delta_H^{-1}\rangle=\langle T|\pi^*(\Delta_H^{-1})\rangle
	\end{align*}
	for every right positive definite $T\in\mathscr{E}'(G)$. Note that for every $x\in G$, the distribution $\delta_x$ on $G$ is equal to $\delta_x*\delta_{e_G}^\flat$ and therefore may be written via the polarisation identity as a linear combination of right positive definite finitely supported distributions. It then follows that $\Delta_G=\pi^*(\Delta_H)$.
	
	Now condition (ii) combined with the relation $\Delta_G=\pi^*(\Delta_H)$ imply that the pushforward by $\pi$ of any left positive definite compactly supported distribution is left positive definite since for every left positive definite $T\in\mathscr{E}'(G)$
	\[
		\Delta_H^{1/2}\pi_*T=\pi_*(\pi^*(\Delta_H^{1/2})T)=\pi_*(\Delta_G^{1/2}T)
	\]
	must be right positive definite since $\Delta_G^{1/2}T$ is. This implies that the kernel of $\pi$ is amenable (\autoref{theorem:pushleftstrict}).
	
	(iii) $\Rightarrow$ (ii). Suppose that the kernel of $\pi$ is amenable and that $\Delta_G=\pi^*(\Delta_H)$. The amenability of the kernel of $\pi$ implies that the pushforward by $\pi$ of any left positive definite compactly supported distribution is left positive definite, and the relation $\Delta_G=\pi^*(\Delta_H)$ then implies that the pushforward by $\pi$ of any right positive definite compactly supported distribution is right positive definite since for every right positive definite $T\in\mathscr{E}'(G)$,
	\[
		\Delta_H^{1/2}\pi_*T=\pi_*(\pi^*(\Delta_H^{1/2})T)=\pi_*(\Delta_G^{1/2}T)
	\]
	must be left positive definite since $\Delta_G^{1/2}T$ is.
\end{proof}
We thus obtain a precise description of when the pushforward by $\pi$ of a compactly supported right positive definite continuous function is a right positive definite continuous function.
\begin{corollary}\label{corollary:pushcspdc}
    Let $G$ and $H$ be two locally compact groups and let $\pi$ be a continuous homomorphism from $G$ to $H$. The following conditions are equivalent:
    \begin{enumerate}
        \item The pushforward by $\pi$ of any right positive definite compactly supported continuous function on $G$ is a right positive definite continuous function on $H$.
        \item $\pi$ is open, its kernel is amenable and the modular function of $G$ is the pullback by $\pi$ of the modular function of $H$.
    \end{enumerate}
\end{corollary}
The condition on the modular functions may seem a little esoteric. In the case of open homomorphisms, the following result shows how restrictive it is by providing a few equivalent conditions.
\begin{proposition}
    Let $G$ and $H$ be two locally compact groups and let $\pi$ be an open continuous homomorphism from $G$ to $H$. The following conditions are equivalent:
    \begin{enumerate}
        \item The modular function of $G$ is the pullback by $\pi$ of the modular function of $H$.
        \item The left Haar measure of the kernel of $\pi$ is invariant under the inner automorphisms of $G$.
        \item For the operation of convolution, the left Haar measure of the kernel of $\pi$ commutes with every compactly supported distribution on $G$.
    \end{enumerate}
\end{proposition}
\begin{proof}
    The equivalence (i)~$\Leftrightarrow$~(ii) follows immediately from a classical formula for the modular function of a quotient group \cite[chap.~VII, \textsection\nobreak\,2, cor. of prop.~11]{BourbakiINTd}. Condition (ii) is none other than condition (iii) in the special case of Dirac distributions. The implication (iii)~$\Rightarrow$~(ii) is therefore obvious. The implication (ii)~$\Rightarrow$~(iii) can then be obtained by a simple linearity, continuity and density argument.
\end{proof}
We end this section with a norm estimate for the conclusion of Jodeit's theorem simpler and more precise than the one provided by \autoref{proposition:openimpliesjodeit} in the special case of when the homomorphism pushes forward the compactly support right positive definite continuous functions into right positive definite continuous functions.
\begin{proposition}
    Let $G$ and $H$ be two locally compact groups, let $\pi$ be a continuous homomorphism from $G$ to $H$ satisfying the equivalent conditions of the \autoref{corollary:pushcspdc} of \autoref{theorem:pushrightstrict} and denote by $N$ its kernel. Then, the conclusion of Jodeit's theorem holds for $\pi$. Moreover, for every compact subset $K$ of $G$ containing $e_G$, every non-negligible compact subset $L$ of $G$ and every $p,q\in[1,\infty]$, it holds that
    \begin{multline*}
        \lVert\pi_*\rVert_{\operatorname{M}_K^{p,q}(H)\to M^{p,q}(H)}\\
        \leq\biggl(\frac{\mu_G(KL)}{\mu_G(L)}\biggr)^{\mkern-5mu\frac 1{2p'}+\frac 1{2q}}\biggl(2\mkern1.5mu\mu_N\bigl(KL(KL)^{-1}\bigr)\biggl(\frac{\mu_G(KL)}{\mu_G(L)}+1\biggr)\biggr)^{\mkern-5mu    \frac 1{p}+\frac 1{q\mathrlap{'}}}.
    \end{multline*}
\end{proposition}
\begin{proof}
    By choosing $\varphi\in\mathscr{D}(G)$ in \autoref{remark:estimate} arbitrarily close to
    \[
        \psi={\boldsymbol 1}_{KL}*(\mu_G^{-1}{\boldsymbol 1}_L)^\flat
    \]
    in $A(G)$, it suffices to show that
    \[
        \lVert f\mapsto\pi_*(\psi f)\rVert_{A(G)\to A(H)}\leq 2\mkern1.5mu\mu_N\bigl(KL(KL)^{-1}\bigr)\biggl(\frac{\mu_G(KL)}{\mu_G(L)}+1\biggr)\lVert f\rVert_{A(G)}.
    \]
    Let $f\in A(G)$ and for every integer $k\in[0,3]$, let
    \[
        \psi_k=\frac 1{4\mkern1.5mu\mu_G(L)}({\boldsymbol 1}_{KL}+i^k{\boldsymbol 1}_L)*({\boldsymbol 1}_{KL}+i^k{\boldsymbol 1}_L)^\flat,
    \]
    so that each $\psi_k$ is right positive definite and $\psi=\psi_0+i\psi_1-\psi_2-i\psi_3$. Similarly, by decomposing $f$ into the positive and negative parts of its real and imaginary parts in the involutive algebra $A(G)$ \cite[n\textsuperscript{o}\nobreak\,1.2 and prop.~3.15]{Eymard}, $f$ can be written as $f_0-f_1+i(f_2-f_3)$ where each $f_k$ is in $A(G)$ and right positive definite and $\sum_{0\leq k\leq 3}f_k(e_G)\leq 2\lVert f\rVert_{A(G)}$. Then, since each $\pi_*(\psi_kf_l)$ is right positive definite, we have
    \[
        \lVert\pi_*(\psi f)\rVert_{A(H)}\leq\sum_{0\leq k,l\leq 3}\lVert\pi_*(\psi_kf_l)\rVert_{A(H)}=\sum_{0\leq k,l\leq 3}\pi_*(\psi_kf_l)(e_H).
    \]
    But then,
    \begin{align*}
        \pi_*(\psi_kf_l)(e_H)&=\pi_*(\psi_kf_l)(\pi(e_G))\\
        &=\pi^*\pi_*(\psi_kf_l)(e_G)=((\psi_kf_l)*(\mu_N)^\flat)(e_G)=\int_N\psi_kf_l\,d\mu_N.
    \end{align*}
    A naïve estimates then yields
    \begin{align*}
        \int_N\psi_kf_l\,d\mu_N&\leq\mu_N(\supp\psi_k)\sup_N\lvert\psi_kf_l\rvert\\
        &\leq\mu_N\bigl(KL(KL)^{-1}\bigr)\psi_k(e_G)f_l(e_G)\\
        &=\mu_N\bigl(KL(KL)^{-1}\bigr)\frac{\lVert{\boldsymbol 1}_{KL}+i^k{\boldsymbol 1}_L\rVert_{L^2(G)}^2}{4\mkern1.5mu\mu_G(L)}f_l(e_G)
    \end{align*}
    and therefore
    \begin{align*}
        \lVert\pi_*(\psi f)\rVert_{A(H)}
        &\leq\sum_{0\leq k,l\leq 3}\mu_N\bigl(KL(KL)^{-1}\bigr)\frac{\lVert{\boldsymbol 1}_{KL}+i^k{\boldsymbol 1}_L\rVert_{L^2(G)}^2}{4\mkern1.5mu\mu_G(L)}f_l(e_G)\\
        &=\mu_N\bigl(KL(KL)^{-1}\bigr)\frac{\lVert{\boldsymbol 1}_{KL}\rVert_{L^2(G)}^2+\lVert{\boldsymbol 1}_{L}\rVert_{L^2(G)}^2}{\mu_G(L)}2\lVert f\rVert_{A(G)}\\
        &=2\mkern1.5mu\mu_N\bigl(KL(KL)^{-1}\bigr)\biggl(\frac{\mu_G(KL)}{\mu_G(L)}+1\biggr)\lVert f\rVert_{A(G)}
    \end{align*}
    as desired.
\end{proof}\backmatter
\section{Acknowledgements}
The author would like to thank Quanhua~Xu for his useful comments and advice and Mikael~de~la~Salle without whom the proof of \autoref{proposition:schurpullback} could have never been completed. The present work was produced as part of the author's doctoral contract funded by the Ministère de l'Enseignement supérieur et de la Recherche.
\biblabelsep1.5\labelsep
\sloppy
\printbibliography[heading=bibintoc]

@article{ArhancetUnconditionality,
 author = {Arhancet, C{\'e}dric},
 title = {Unconditionality, {Fourier} multipliers and {Schur} multipliers},
 fjournal = {Colloquium Mathematicum},
 journal = {Colloq. Math.},
 issn = {0010-1354},
 volume = {127},
 number = {1},
 pages = {17--37},
 year = {2012},
 language = {english},
 doi = {10.4064/cm127-1-2},
 zbMATH = {6049674},
 Zbl = {1253.43001}
}

@article{AuscherCarro,
 author = {Auscher, Pascal and Carro, Mar{\'i}a J.},
 title = {On relations between operators on {{\(\mathbf{R}^ {N}\)}}, {{\(\mathbf{T}^ {N}\)}} and {{\(\mathbf{Z}^ {N}\)}}},
 fjournal = {Studia Mathematica},
 journal = {Stud. Math.},
 issn = {0039-3223},
 volume = {101},
 number = {2},
 pages = {165--182},
 year = {1992},
 language = {English},
 doi = {10.4064/sm-101-2-165-182},
 zbMATH = {711680},
 Zbl = {0810.42004}
}

@book{BerghLofstrom,
 author = {Bergh, J{\"o}ran and L{\"o}fstr{\"o}m, J{\"o}rgen},
 title = {Interpolation spaces. {An} introduction},
 fseries = {Grundlehren der Mathematischen Wissenschaften},
 series = {Grundlehren Math. Wiss.},
 issn = {0072-7830},
 volume = {223},
 year = {1976},
 language = {english},
 publisher = {Springer, Cham},
 isbn = {978-3-642-66453-3},
 doi = {10.1007/978-3-642-66451-9},
 zbMATH = {3536702},
 Zbl = {0344.46071}
}

@Book{BourbakiINTd,
 Author = {Bourbaki, Nicolas},
 Title = {{\'E}l{\'e}ments de math{\'e}matique. {Int{\'e}gration}. {Chapitres} 7 et 8},
 Edition = {R{\'e}impression inchang{\'e}e de l'{\'e}dition originale de 1963},
 ISBN = {978-3-540-35324-9},
 language = {french},
 langid = {french},
 Year = {2007},
 Publisher = {Berlin: Springer},
 DOI = {10.1007/978-3-540-35325-6},
 zbMATH = {5080977},
 Zbl = {1106.46005}
}

@article{BozejkoFendler,
 author = {Bo{\.z}ejko, Marek and Fendler, Gero},
 title = {Herz-Schur multipliers and completely bounded multipliers of the {Fourier} algebra of a locally compact group},
 fjournal = {Bollettino della Unione Matem{\`a}tica Italiana. Serie VI. A},
 journal = {Boll. Unione Mat. Ital., VI. Ser., A},
 issn = {0392-4033},
 volume = {3},
 pages = {297--302},
 language = {english},
 year = {1984},
 zbMATH = {3899517},
 Zbl = {0564.43004}
}

@Article{Bruhat,
 Author = {Bruhat, Fran{\c{c}}ois},
 Title = {Distributions sur un groupe localement compact et applications {\`a} l'{\'e}tude des repr{\'e}sentations des groupes {{\(p\)}}-adiques},
 FJournal = {Bulletin de la Soci{\'e}t{\'e} Math{\'e}matique de France},
 Journal = {Bull. Soc. Math. Fr.},
 ISSN = {0037-9484},
 Volume = {89},
 Pages = {43--75},
 language = {french},
 Year = {1961},
 DOI = {10.24033/bsmf.1559},
 zbMATH = {3209560},
 Zbl = {0128.35701}
}

@article{Caspers,
 author = {Caspers, Martijn},
 title = {The {{\(L^p\)}}-{Fourier} transform on locally compact quantum groups},
 fjournal = {Journal of Operator Theory},
 journal = {J. Oper. Theory},
 issn = {0379-4024},
 volume = {69},
 number = {1},
 pages = {161--193},
 language = {english},
 year = {2013},
 doi = {10.7900/jot.2010aug22.1949},
 zbMATH = {6160738},
 Zbl = {1313.43003}
}

@article{CaspersJanssensKrishnaswayUshaMiaskiwskyi,
 author = {Caspers, Martijn and Janssens, Bas and Krishnaswamy-Usha, Amudhan and Miaskiwskyi, Lukas},
 title = {Local and multilinear noncommutative de {Leeuw} theorems},
 fjournal = {Mathematische Annalen},
 journal = {Math. Ann.},
 issn = {0025-5831},
 volume = {388},
 number = {4},
 pages = {4251--4305},
 year = {2024},
 language = {English},
 doi = {10.1007/s00208-023-02611-z},
 zbMATH = {7826797},
 Zbl = {1540.22017}
}

@Article{CaspersParcetPerrinRicard,
 Author = {Caspers, Martijn and Parcet, Javier and Perrin, Mathilde and Ricard, {\'E}ric},
 Title = {Noncommutative de {Leeuw} theorems},
 FJournal = {Forum of Mathematics, Sigma},
 Journal = {Forum Math. Sigma},
 ISSN = {2050-5094},
 Volume = {3},
 Note = {No. e21},
 language = {english},
 Year = {2015},
 DOI = {10.1017/fms.2015.23},
 zbMATH = {6502815},
 Zbl = {1373.42010}
}

@article{CaspersdelaSalle,
 author = {Caspers, Martijn and de la Salle, Mikael},
 title = {Schur and {Fourier} multipliers of an amenable group acting on non-commutative {{\(L^{p}\)}}-spaces},
 fjournal = {Transactions of the American Mathematical Society},
 journal = {Trans. Am. Math. Soc.},
 issn = {0002-9947},
 volume = {367},
 number = {10},
 pages = {6997--7013},
 language = {english},
 year = {2015},
 doi = {10.1090/S0002-9947-2015-06281-3},
 zbMATH = {6479348},
 Zbl = {1325.43002}
}

@article{CowlingExtensionperiodicity,
 author = {Cowling, Michael G.},
 title = {Extension of multipliers by periodicity},
 fjournal = {Bulletin of the Australian Mathematical Society},
 journal = {Bull. Aust. Math. Soc.},
 issn = {0004-9727},
 volume = {6},
 pages = {263--285},
 language = {english},
 year = {1972},
 doi = {10.1017/S000497270004449X},
 zbMATH = {3362402},
 Zbl = {0228.43008}
}

@Article{CowlingExtensionLpLq,
 Author = {Cowling, Michael G.},
 Title = {Extension of {Fourier} {{\(L^p\)-\(L^q\)}} multipliers},
 FJournal = {Transactions of the American Mathematical Society},
 Journal = {Trans. Am. Math. Soc.},
 ISSN = {0002-9947},
 Volume = {213},
 Pages = {1--33},
 language = {english},
 Year = {1975},
 DOI = {10.2307/1998031},
 zbMATH = {3486424},
 Zbl = {0311.43007}
}

@Article{CwikelJanson,
 Author = {Cwikel, Michael and Janson, Svante},
 Title = {Interpolation of analytic families of operators},
 FJournal = {Studia Mathematica},
 Journal = {Stud. Math.},
 ISSN = {0039-3223},
 Volume = {79},
 Pages = {61--71},
 language = {english},
 Year = {1984},
 DOI = {10.4064/sm-79-1-61-71}
}

@book{DixmierCetoile,
 author = {Dixmier, Jacques},
 title = {Les {{\({C}^ *\)}}-alg{\`e}bres et leurs repr{\'e}sentations},
 edition = {2\textsuperscript{e} éd.},
 language = {french},
 year = {1969},
 publisher = {{Paris}: {Gauthier}-{Villars}},
 series = {Cahiers scientifiques. Fascicule XXIX},
 zbMATH = {3277956},
 Zbl = {0174.18601}
}

@Article{Eymard,
 Author = {Eymard, Pierre},
 Title = {L'alg{\`e}bre de {Fourier} d'un groupe localement compact},
 FJournal = {Bulletin de la Soci{\'e}t{\'e} Math{\'e}matique de France},
 Journal = {Bull. Soc. Math. Fr.},
 ISSN = {0037-9484},
 Volume = {92},
 Pages = {181--236},
 language = {french},
 Year = {1964},
 DOI = {10.24033/bsmf.1607},
 zbMATH = {3271940},
 Zbl = {0169.46403}
}

@article{GaudryQuasimeasures,
 author = {Gaudry, Garth I.},
 title = {Quasimeasures and operators commuting with convolution},
 fjournal = {Pacific Journal of Mathematics},
 journal = {Pac. J. Math.},
 issn = {1945-5844},
 volume = {18},
 pages = {461--476},
 language = {english},
 year = {1966},
 doi = {10.2140/pjm.1966.18.461},
 zbMATH = {3250987},
 Zbl = {0156.14601}
}

@article{Godement,
 author = {Godement, Roger},
 title = {Les fonctions de type positif et la th{\'e}orie des groupes},
 fjournal = {Transactions of the American Mathematical Society},
 journal = {Trans. Am. Math. Soc.},
 issn = {0002-9947},
 volume = {63},
 pages = {1--84},
 language = {french},
 year = {1948},
 doi = {10.2307/1990636},
 zbMATH = {3049017},
 Zbl = {0031.35903}
}

@article{IzumiConstruction,
 author = {Izumi, Hideaki},
 title = {Constructions of non-commutative {{\(L^p\)}}-spaces with a complex parameter arising from modular actions},
 fjournal = {International Journal of Mathematics},
 journal = {Int. J. Math.},
 issn = {0129-167X},
 volume = {8},
 number = {8},
 pages = {1029--1066},
 language = {english},
 year = {1997},
 doi = {10.1142/S0129167X97000494},
 zbMATH = {1127548},
 Zbl = {0904.46046}
}

@article{IzumiBilin,
 author = {Izumi, Hideaki},
 title = {Natural bilinear forms, natural sesquilinear forms and the associated duality on non-commutative {{\(L^p\)}}-spaces},
 fjournal = {International Journal of Mathematics},
 journal = {Int. J. Math.},
 issn = {0129-167X},
 volume = {9},
 number = {8},
 pages = {975--1039},
 language = {english},
 year = {1998},
 doi = {10.1142/S0129167X98000439},
 zbMATH = {1267830},
 Zbl = {0934.46061}
}

@article{JanssensOudejans,
 author = {Janssens, Bas and Oudejans, Benjamin},
 title = {Local noncommutative {De} {Leeuw} theorems beyond reductive {Lie} groups},
 fjournal = {Journal of Lie Theory},
 journal = {J. Lie Theory},
 issn = {0949-5932},
 volume = {35},
 number = {4},
 pages = {845--860},
 year = {2025},
 language = {English},
 url = {www.heldermann.de/JLT/JLT35/JLT354/jlt35041.htm#jlt354},
 zbMATH = {8124773}
}

@Article{Jodeit,
 Author = {family=Jodeit, suffix=Jr., given=Max},
 Title = {Restrictions and extensions of {Fourier} multipliers},
 FJournal = {Studia Mathematica},
 Journal = {Stud. Math.},
 ISSN = {0039-3223},
 Volume = {34},
 Pages = {215--226},
 language = {english},
 Year = {1970},
 DOI = {10.4064/sm-34-2-215-226},
 zbMATH = {3317449},
 Zbl = {0199.20302}
}

@Article{LafforguedelaSalle,
 Author = {Lafforgue, Vincent and prefix={de la}, family=Salle, given=Mikael},
 Title = {Noncommutative {{\(L^{p}\)}}-spaces without the completely bounded approximation property},
 FJournal = {Duke Mathematical Journal},
 Journal = {Duke Math. J.},
 ISSN = {0012-7094},
 Volume = {160},
 Number = {1},
 Pages = {71--116},
 language = {english},
 Year = {2011},
 DOI = {10.1215/00127094-1443478},
 zbMATH = {5986340},
 Zbl = {1267.46072}
}

@Article{deLeeuw,
 Author = {prefix=de, family=Leeuw, given=Karel},
 Title = {On {{\(L_ p\)}} multipliers},
 FJournal = {Annals of Mathematics. Second Series},
 Journal = {Ann. Math. (2)},
 ISSN = {0003-486X},
 Volume = {81},
 Pages = {364--379},
 language = {english},
 Year = {1965},
 DOI = {10.2307/1970621},
 zbMATH = {3272605},
 Zbl = {0171.11803}
}

@article{LohoueApprox,
 author = {Lohou{\'e}, No{\"e}l},
 title = {Approximation et transfert d'op{\'e}rateurs de convolution},
 fjournal = {Annales de l'Institut Fourier},
 journal = {Ann. Inst. Fourier},
 issn = {0373-0956},
 volume = {26},
 number = {4},
 pages = {133--150},
 language = {french},
 year = {1976},
 doi = {10.5802/aif.635},
 zbMATH = {3517691},
 Zbl = {0331.43008}
}

@article{LohoueEnssynthese,
 author = {Lohoué, Noël},
 title = {Sur certains ensembles de synthèse dans les alg{\`e}bres {{\(A_{p}(G)\)}}},
 fjournal = {Comptes Rendus Hebdomadaires des S{\'e}ances de l'Acad{\'e}mie des Sciences, S{\'e}rie A},
 journal = {C. R. Acad. Sci., Paris, S{\'e}r. A},
 issn = {0366-6034},
 volume = {270},
 pages = {589--591},
 language = {french},
 year = {1970},
 url = {https://gallica.bnf.fr/ark:/12148/bpt6k480298g/f596.item},
 zbMATH = {3300214},
 Zbl = {0188.20302}
}

@article{LohoueLasynthese,
 author = {Lohou{\'e}, Noël},
 title = {La synth{\`e}se des convoluteurs sur un groupe ab{\'e}lien localement compact},
 fjournal = {Comptes Rendus Hebdomadaires des S{\'e}ances de l'Acad{\'e}mie des Sciences, S{\'e}rie A},
 journal = {C. R. Acad. Sci., Paris, S{\'e}r. A},
 issn = {0366-6034},
 volume = {272},
 pages = {27--29},
 year = {1971},
 language = {French},
 url = {https://gallica.bnf.fr/ark:/12148/bpt6k480300n/f30.item},
 zbMATH = {3348732},
 Zbl = {0219.43011}
}

@article{NeuwirthRicard,
 author = {Neuwirth, Stefan and Ricard, {\'E}ric},
 title = {Transfer of {Fourier} multipliers into {Schur} multipliers and sumsets in a discrete group},
 fjournal = {Canadian Journal of Mathematics},
 journal = {Can. J. Math.},
 issn = {0008-414X},
 volume = {63},
 number = {5},
 pages = {1161--1187},
 language = {english},
 year = {2011},
 doi = {10.4153/CJM-2011-053-9},
 zbMATH = {5955512},
 Zbl = {1251.47036}
}

@article{ParcetdelaSalleTablate,
 author = {Parcet, Javier and de la Salle, Mikael and Tablate, Eduardo},
 title = {The local geometry of idempotent {Schur} multipliers},
 fjournal = {Forum of Mathematics, Pi},
 journal = {Forum Math. Pi},
 issn = {2050-5086},
 volume = {13},
 pages = {21},
 note = {Id/No e14},
 language = {english},
 year = {2025},
 doi = {10.1017/fmp.2025.6},
 zbMATH = {8019218}
}

@Book{Pier,
 author = {Pier, Jean-Paul},
 title = {Amenable locally compact groups},
 language = {english},
 year = {1984},
 ISBN = {0-471-89390-0},
 Series = {Pure and Applied Mathematics. A Wiley-Interscience Publication},
 Publisher = {John Wiley \& Sons},
 zbMATH = {4006928},
 Zbl = {0621.43001}
}

@Book{PisierIntrotoOS,
 Author = {Pisier, Gilles},
 Title = {Introduction to operator space theory},
 FSeries = {London Mathematical Society Lecture Note Series},
 Series = {Lond. Math. Soc. Lect. Note Ser.},
 ISSN = {0076-0552},
 Volume = {294},
 ISBN = {0-521-81165-1},
 language = {english},
 Year = {2003},
 Publisher = {Cambridge: Cambridge University Press},
 zbMATH = {1993745},
 Zbl = {1093.46001}
}

@book{PisierNCvvLp,
 author = {Pisier, Gilles},
 title = {Non-commutative vector valued {{\(L_p\)}}-spaces and completely {{\(p\)}}-summing maps},
 fseries = {Ast{\'e}risque},
 series = {Ast{\'e}risque},
 issn = {0303-1179},
 volume = {247},
 language = {english},
 year = {1998},
 publisher = {Paris: Soci{\'e}t{\'e} Math{\'e}matique de France},
 zbMATH = {1216173},
 Zbl = {0937.46056}
}

@Article{Saeki,
 Author = {Saeki, Sadahiro},
 Title = {Translation invariant operators on groups},
 FJournal = {T{\^o}hoku Mathematical Journal. Second Series},
 Journal = {T{\^o}hoku Math. J. (2)},
 ISSN = {0040-8735},
 Volume = {22},
 Pages = {409--419},
 language = {english},
 Year = {1970},
 DOI = {10.2748/tmj/1178242767},
 zbMATH = {3327654},
 Zbl = {0206.12601}
}

@book{SchwartzRadon,
 author = {Schwartz, Laurent},
 title = {Radon measures on arbitrary topological spaces and cylindrical measures},
 fseries = {Tata Institute of Fundamental Research. Studies in Mathematics},
 series = {Tata Inst. Fundam. Res., Stud. Math.},
 volume = {6},
 year = {1973},
 publisher = {London: Oxford University Press, published for the Tata Institute of Fundamental Research},
 language = {English},
 zbMATH = {3467467},
 Zbl = {0298.28001}
}

@Book{SchwartzTD,
 Author = {Schwartz, Laurent},
 Title = {Th{\'e}orie des distributions},
 Edition = {Nouvelle {\'e}dition, enti{\`e}rement corrig{\'e}e, refondue et augment{\'e}e},
 fseries = {Publications de l'Institut de Math{\'e}matique de l'Universit{\'e} de Strasbourg},
 series = {Publ. Inst. Math. Strasbourg},
 volume = {9-10},
 language = {french},
 Year = {1966},
 Publisher = {Hermann, Paris},
 zbMATH = {3240665},
 Zbl = {0149.09501}
}

@article{Wawrzynczyk,
 author = {Wawrzy{\'n}czyk, Antoni},
 title = {On tempered distributions and {Bochner}{-}{Schwartz} theorem on arbitrary locally compact abelian groups},
 fjournal = {Colloquium Mathematicum},
 journal = {Colloq. Math.},
 issn = {0010-1354},
 volume = {19},
 pages = {305--318},
 year = {1968},
 language = {English},
 doi = {10.4064/cm-19-2-305-318},
 zbMATH = {3293320},
 Zbl = {0184.17501}
}
\vfill
\end{document}